\documentclass[10pt]{amsart}

\usepackage[utf8]{inputenc}
\usepackage[T1]{fontenc}
\usepackage[backend=biber,
            isbn=false,
            doi=false,
            maxbibnames=5,
            giveninits=true,
            style=alphabetic,
            citestyle=alphabetic]{biblatex}
\usepackage{url}
\renewbibmacro{in:}{}
\bibliography{quantummomentmap.bib}

\usepackage{color}
\definecolor{e-mail}{rgb}{0,.40,.80}
\definecolor{reference}{rgb}{.20,.60,.22}
\definecolor{citation}{rgb}{0,.40,.80}

\usepackage[colorlinks=true,
            linkcolor=reference,
            citecolor=citation,
            urlcolor=e-mail]{hyperref}
\usepackage[all]{xy}
\usepackage{amssymb}
\usepackage{eucal}
\usepackage{cleveref}
\usepackage{stmaryrd}
\usepackage{fullpage}

\newcommand{\cA}{\mathcal{A}}
\newcommand{\cB}{\mathcal{B}}
\newcommand{\cC}{\mathcal{C}}
\newcommand{\cD}{\mathcal{D}}
\newcommand{\cE}{\mathcal{E}}
\newcommand{\cF}{\mathcal{F}}
\newcommand{\cM}{\mathcal{M}}
\newcommand{\cN}{\mathcal{N}}
\newcommand{\cO}{\mathcal{O}}

\newcommand{\fd}{\mathfrak{d}}
\newcommand{\g}{\mathfrak{g}}
\newcommand{\h}{\mathfrak{h}}
\newcommand{\n}{\mathfrak{n}}

\newcommand{\B}{\mathrm{B}}
\renewcommand{\d}{\mathrm{d}}
\newcommand{\D}{\mathrm{D}}
\newcommand{\T}{\mathrm{T}}
\newcommand{\U}{\mathrm{U}}
\newcommand{\Z}{\mathrm{Z}}

\newcommand{\C}{\mathbf{C}}
\newcommand{\br}{\mathbf{r}}

\newcommand{\bE}{\mathbb{E}}

\newcommand{\calpha}{{\check\alpha}}
\newcommand{\cLambda}{{\check\Lambda}}

\newcommand{\Alg}{\mathrm{Alg}}
\newcommand{\bop}{\sigma\mathrm{op}}
\newcommand{\BrTens}{\mathrm{BrTens}}
\newcommand{\CoAlg}{\mathrm{CoAlg}}
\newcommand{\colim}{\mathrm{colim}}
\newcommand{\CoMod}{\mathrm{CoMod}}
\newcommand{\uEnd}{\underline{\mathrm{End}}}
\newcommand{\ev}{\mathrm{ev}}
\newcommand{\Fun}{\mathrm{Fun}}
\newcommand{\FunL}{\mathrm{Fun}^\mathrm{L}}
\newcommand{\HC}{\mathcal{HC}}
\newcommand{\Hom}{\mathrm{Hom}}
\newcommand{\id}{\mathrm{id}}
\newcommand{\inv}{\mathrm{inv}}
\newcommand{\LagrCorr}{\mathrm{LagrCorr}}
\newcommand{\Lie}{\mathrm{Lie}}
\newcommand{\LMod}{\mathrm{LMod}}
\newcommand{\Mod}{\mathrm{Mod}}
\newcommand{\mop}{\otimes\mathrm{op}}
\newcommand{\cPrL}{\mathcal{P}\mathrm{r}^{\mathrm{L}}}
\newcommand{\PrL}{\mathrm{Pr}^{\mathrm{L}}}
\newcommand{\pt}{\mathrm{pt}}
\newcommand{\QCoh}{\mathrm{QCoh}}
\newcommand{\Rep}{\operatorname{Rep}}
\newcommand{\RMod}{\mathrm{RMod}}
\newcommand{\Sym}{\mathrm{Sym}}
\newcommand{\triv}{\mathrm{triv}}
\newcommand{\Vect}{\mathrm{Vect}}

\newtheorem{thm}{Theorem}[section]
\newtheorem*{theorem}{Theorem}
\newtheorem{prop}[thm]{Proposition}

\newtheorem{lm}[thm]{Lemma}
\theoremstyle{definition}
\newtheorem{defn}[thm]{Definition}
\theoremstyle{remark}
\newtheorem{remark}[thm]{Remark}
\newtheorem{example}[thm]{Example}

\newcommand{\BiMod}[2]{{}_{#1}\mathrm{BiMod}_{#2}}

\newcommand{\defterm}[1]{\textbf{\emph{#1}}}

\newcommand{\bicoeq}[3]{
\xymatrix{
#1 & #2 \ar@<.5ex>[l] \ar@<-.5ex>[l] & #3 \ar@<.8ex>[l] \ar[l] \ar@<-.8ex>[l] 
}
}

\newcommand{\simp}[2]{
\xymatrix{
#1 & #2 \ar@<.5ex>[l] \ar@<-.5ex>[l] & \dots \ar@<.8ex>[l] \ar[l] \ar@<-.8ex>[l] 
}
}

\newcommand{\cosimp}[2]{
\xymatrix{
#1 \ar@<.5ex>[r] \ar@<-.5ex>[r] & #2 \ar@<.8ex>[r] \ar[r] \ar@<-.8ex>[r] & \dots
}
}

\begin{document}
\title{A categorical approach to quantum moment maps}
\address{Institut f\"{u}r Mathematik, Universit\"{a}t Z\"{u}rich, Zurich, Switzerland}
\email{p.safronov@ed.ac.uk}
\author{Pavel Safronov}
\begin{abstract}
We introduce quantum versions of Manin pairs and Manin triples and define quantum moment maps in this context. This provides a framework that incorporates quantum moment maps for actions of Lie algebras and quantum groups for any quantum parameter. We also show how our quantum moment maps degenerate to known classical versions of moment maps and describe their fusion.
\end{abstract}
\maketitle

\section*{Introduction}

In this paper we define quantum Manin pairs and quantum Manin triples and describe quantum moment maps in this setting.

\subsection*{Moment maps}

Given a Poisson manifold $X$ with an action of a Lie group $G$ that preserves a Poisson structure, a moment map (see \cite[Chapter 11]{MarsdenRatiu}) is a $G$-equivariant map $\mu\colon X\rightarrow \g^*$ which gives the Hamiltonian for the infinitesimal $\g$-action on $X$. This concept goes back to the works of Kostant, Souriau, Marsden and Weinstein. If $G$ is compact and it acts freely on $\mu^{-1}(0)$, we have the reduced space $X//G = \mu^{-1}(0) / G$ which is still a Poisson manifold.

In problems related to quantum groups, the group $G$ usually does not preserve the Poisson structure on $X$. Instead, $G$ is a Poisson-Lie group, i.e. it carries a multiplicative Poisson structure, and the action map $G\times X\rightarrow X$ is Poisson. The first theory of moment maps in this setting was proposed by Lu and Weinstein (see \cite{LuClassical} and \cite{LuWeinstein}) where the moment maps are maps $\mu\colon X\rightarrow G^*$ satisfying certain conditions, where $G^*$ is the Poisson-Lie dual group. For instance, if $G$ carries the zero Poisson-Lie structure, we may take $\g^*$ as the Poisson-Lie dual group thought of as an abelian group under addition equipped with the Kirillov--Kostant--Souriau Poisson structure. So, in this case the theory reduces to ordinary moment maps $\mu\colon X\rightarrow \g^*$.

To provide a finite-dimensional description of the symplectic structure on character varieties, Alekseev, Malkin and Meinrenken \cite{AMM} introduced a version of moment maps $\mu\colon X\rightarrow \g^*$ for symplectic manifolds. Namely, given a nondegenerate pairing $c\in\Sym^2(\g^*)^G$ they consider group-valued moment maps $\mu\colon X\rightarrow G$ together with a certain quasi-symplectic structure on $X$ satisfying certain conditions. They also prove that under standard assumptions $X//G = \mu^{-1}(e) / G$ is a symplectic manifold.

Generalizing Lu--Weinstein moment maps, Alekseev and Kosmann-Schwarzbach \cite{AKS} defined moment maps for actions of quasi-Poisson groups $G$. These are moment maps $\mu\colon X\rightarrow D/G$, where $D$ is the double Lie group of the quasi-Poisson group $G$. Moreover, in the Poisson-Lie case there is a natural morphism $G^*\rightarrow D/G$ and the Alekseev--Kosmann-Schwarzbach notion of moment maps factoring through $G^*$ recovers Lu's notion. If $G$ is equipped with a nondegenerate pairing $c$, we have the double $D = G\times G$ where $G\subset D$ is embedded diagonally. So, $D/G\cong G$ and in this case we get moment maps $\mu\colon X\rightarrow G$. It was shown in \cite{AKSM} that under a further nondegeneracy assumption on the quasi-Poisson manifold $X$, the resulting theory is equivalent to quasi-symplectic group-valued moment maps of Alekseev--Malkin--Meinrenken.

\subsection*{Quantum moment maps}

Quantum analogs of the moment maps $X\rightarrow \g^*$ are given as follows. One considers an algebra $A$ equipped with a compatible action of an algebraic group $G$ and a $G$-equivariant map $\mu\colon \U\g\rightarrow A$ such that $[\mu(x), -]$ is the infinitesimal action of $x\in\g$ on $A$. These moment maps are ubiquitous in the physics literature on the quantum BRST method.

Quantum analogs of Lu--Weinstein moment maps were introduced in \cite{LuQuantum}. Consider a Hopf algebra $H$ and an $H$-module algebra $A$. A moment map is then a map of algebras $\mu\colon H\rightarrow A$ such that $h\triangleright a = \mu(h_{(1)}) a \mu(S(h_{(2)}))$ for every $h\in H$ and $a\in A$, where $h\triangleright a$ denotes the $H$-action on $A$. Here we think of $H$ as the quantization of the Poisson-Lie dual group $G^*$.

A variant of Lu's quantum moment maps was proposed by Varagnolo and Vasserot in \cite{VaragnoloVasserot}. There one considers a left $H$-coideal subalgebra $H'\subset H$, i.e. the coproduct on $H$ restricts on $H'$ to a coaction $H'\rightarrow H\otimes H'$. Then a quantum moment map is an algebra map $\mu\colon H'\rightarrow A$ such that $\mu(h)a = h_{(1)}\triangleright a \cdot \mu(h_{(2)})$. Recall that the Alekseev--Kosmann-Schwarzbach moment maps factoring as $X\rightarrow G^*\rightarrow D/G$ are the same as Lu's classical moment maps. On the quantum level Varagnolo--Vasserot definition for $H'=H$ also recovers Lu's quantum moment maps. In applications to quantum groups (see e.g. \cite{Jordan} and \cite{JordanBalagovich}), one takes $H=\U_q(\g)$, the quantum group associated to a Lie algebra $\g$, and $H'\subset H$ as the reflection equation algebra $\cO_q(G)\subset \U_q(\g)$ (see \cite{KulishSklyanin}, \cite{MajidBraided} and \cite{KolbStokman} for its definition). In particular, in \cite[Section 3.1]{Jordan} it is explicitly suggested that the classical limit of the Varagnolo--Vasserot moment maps are the Alekseev--Kosmann-Schwarzbach moment maps.

\subsection*{Shifted Poisson geometry}

The goal of the present paper is to provide a comprehensive study of moment maps on the quantum level and show that they recover Alekseev--Kosmann-Schwarzbach notion after classical degeneration. Our results and constructions are heavily inspired by the theory of shifted symplectic structures \cite{PTVV} and shifted Poisson structures \cite{CPTVV}. Let us briefly explain how to understand previous constructions from this point of view.

One may organize $n$-shifted symplectic stacks into the following symmetric monoidal 2-category $\LagrCorr_n$:
\begin{itemize}
\item Its objects are $n$-shifted symplectic stacks.

\item 1-morphisms from an $n$-shifted symplectic stack $X$ to an $n$-shifted symplectic stack $Y$ are given by a Lagrangian correspondence $X\leftarrow L\rightarrow Y$, i.e. by an $n$-shifted Lagrangian map $L\rightarrow \overline{X}\times Y$, where $\overline{X}$ denotes the stack $X$ with the opposite $n$-shifted symplectic structure.

\item 2-morphisms from $X\leftarrow L_1\rightarrow Y$ to $X\leftarrow L_2\rightarrow Y$ are given by homotopy classes of stacks $M$ equipped with an $(n-1)$-shifted Lagrangian map $M\rightarrow L_1\times_{\overline{X}\times Y} L_2$.
\end{itemize}
Such a 2-category was constructed by Amorim and Ben-Bassat \cite{ABB} and it is extended to an $(\infty, m)$-category (for any $m$) in the upcoming work of Calaque, Haugseng and Scheimbauer. The unit of $\LagrCorr_n$ is given by the point $\pt$ and $\LagrCorr_n$ has duals and adjoints described as follows:
\begin{itemize}
\item The dual of an $n$-shifted symplectic stack is the same stack equipped with the opposite $n$-shifted symplectic structure.

\item The Lagrangian correspondence $X\leftarrow L\rightarrow Y$ admits a left and right adjoint given by the Lagrangian correspondence $Y\leftarrow L\rightarrow X$.
\end{itemize}

Given a group pair $(D, G)$, it is shown in \cite[Proposition 4.16]{PoissonLie} that the map on classifying stacks $\B G\rightarrow \B D$ has a 2-shifted Lagrangian structure. In other words, a group pair gives rise to a 1-morphism $\pt\rightarrow \B D$ in $\LagrCorr_2$ encoding $\B G$.

We may then consider the 1-shifted symplectic stack
\[\B G\times_{\B D} \B G\cong [G\backslash D/G]\]
obtained as the composite $\pt\rightarrow \B D\rightarrow \pt$ (of $\B G$ and its adjoint) and define classical moment maps to be 1-shifted Lagrangian morphisms $L\rightarrow [G\backslash D/G]$. This definition of classical moment maps is motivated by the following results:
\begin{itemize}
\item As shown by Bursztyn and Crainic \cite{BursztynCrainic}, the quotient $D/G$ carries a natural exact Dirac structure encoding the 1-shifted symplectic stack $[G\backslash D/G]$. Moreover, one can reinterpret quasi-Poisson moment maps of Alekseev--Kosmann-Schwarzbach in terms of Dirac morphisms to $D/G$.

\item It is shown by Calaque \cite{Calaque} (see also \cite{QuasiHamiltonian}) that moment maps $X\rightarrow \g^*$ can be encoded in terms of 1-shifted Lagrangian morphisms $[X/G]\rightarrow [\g^*/G]$ and quasi-symplectic group-valued moment maps $X\rightarrow G$ can be encoded in terms of 1-shifted Lagrangian morphisms $[X/G]\rightarrow [G/G]$. These two cases correspond to the group pairs $(T^* G, G)$ and $(G\times G, G)$ respectively.
\end{itemize}

Given a group triple $(D, G, G^*)$, it is shown in \cite[Proposition 4.17]{PoissonLie} that we have an iterated Lagrangian correspondence
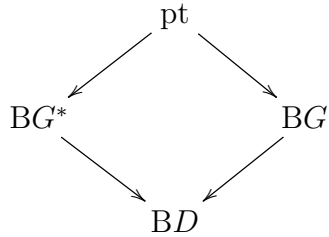
\[
\xymatrix{
& \pt \ar[dl] \ar[dr] & \\
\B G^* \ar[dr] && \B G \ar[dl] \\
& \B D
}
\]
More explicitly, we have a 2-shifted symplectic structure on $\B D$, 2-shifted Lagrangian structures on $\B G\rightarrow \B D$ and $\B G^*\rightarrow \B D$ and a 1-shifted Lagrangian structure on
\[\pt\rightarrow \B G^*\times_{\B D} \B G\cong [G^*\backslash D/G].\]
This may be interpreted within the 2-category $\LagrCorr_2$ in terms of the following data:
\begin{itemize}
\item An object $\B D\in\LagrCorr_2$.

\item 1-morphisms $f\colon 1\rightarrow \B D$ and $g\colon \B D\rightarrow 1$ encoding $\B G$ and $\B G^*$.

\item A 2-morphism $g\circ f\Rightarrow \id_1$ encoding the 1-shifted Lagrangian $\pt\rightarrow [G^*\backslash D/G]$
\end{itemize}

\subsection*{Quantum Manin pairs}

The previous definitions of Manin pairs and Manin triples may be phrased in any (pointed) 2-category. Given an $n$-shifted symplectic stack $X$, one may consider $\bE_n$-monoidal deformations of the symmetric monoidal category $\QCoh(X)$. Given an $n$-shifted Lagrangian map $L\rightarrow X$, its deformation quantization is a pair $(\cC, \cD)$ where $\cC$ is an $\bE_n$-monoidal deformation of $\QCoh(X)$ and $\cD$ is an $\bE_{n-1}$-monoidal deformation of $\QCoh(L)$ together with an action of $\cC$ on $\cD$. Note that an $\bE_2$-monoidal category is the same as a braided monoidal category. So, we may define quantum Manin pairs and quantum Manin triples by replacing the 2-category $\LagrCorr_2$ by $\BrTens$, the 2-category of braided monoidal categories defined as follows:
\begin{itemize}
\item Its objects are braided monoidal categories.

\item 1-morphisms from a braided monoidal category $\cC_1$ to a braided monoidal category $\cC_2$ is a monoidal category $\cD$ equipped with compatible right $\cC_1$- and left $\cC_2$-actions.

\item 2-morphisms from $\cD_1$ to $\cD_2$ which are both equipped with a right $\cC_1$- and left $\cC_2$-action are equivalence classes of $(\cD_2, \cD_1)$-bimodule categories.
\end{itemize}
We refer to \cite[Definition 1.2]{BJS} for a precise description of the 4-category of braided monoidal categories (following previous works \cite{Haugseng}, \cite{Scheimbauer}, \cite{JFS}), so that the 2-category $\BrTens$ is obtained from this 4-category by taking the homotopy 2-category.

Unpacking, a \emph{quantum Manin pair} (see \cref{def:quantumManinpair}) is a pair $(\cC, \cD)$ consisting of a braided monoidal category $\cC$ acting in a compatible way on a monoidal category $\cD$ via a monoidal functor $T\colon \cC\rightarrow \cD$. The right adjoint $T^R\colon \cD\rightarrow \cC$ is lax monoidal and in this setting the algebra $T^R(1)\in\cC$ is commutative (see \cref{prop:REAcommutative}). For an algebra $A\in\cD$ we define the quantum moment map to be an algebra map $\mu\colon \cF=TT^R(1)\rightarrow A$ such that its adjoint $T^R(1)\rightarrow T^R(A)$ is a central map of algebras in $\cC$. A closely related formalism on the classical level has previously appeared in \cite{SeveraCenters}.

It is instructive to consider the following example. For a closed subgroup $G\subset D$ we have a quantum Manin pair $(\Rep D, \Rep G)$ where both categories are symmetric monoidal and $T\colon \Rep D\rightarrow \Rep G$ is the symmetric monoidal restriction functor. Then $T^R(1)$ is the algebra $\cO(D/G)\in\Rep D$ and so a quantum moment map is a map $\cO(D/G)\rightarrow A$ of $G$-representations.

Note that the notion of a 1-shifted Lagrangian $L\rightarrow [G\backslash D/G]$ can also be interpreted purely within the 2-category $\LagrCorr_2$. Its quantization (i.e. the corresponding notion in $\BrTens$) is therefore given as follows. Consider the monoidal category $\HC=\cD\otimes_\cC \cD$ quantizing $[G\backslash D/G]$. Then the quantization of a 1-shifted Lagrangian $L\rightarrow [G\backslash D/G]$ is given by a module category over $\HC$. We show that if $\mu\colon TT^R(1)\rightarrow A$ is a quantum moment map in the above sense, then $\LMod_A$ indeed becomes an $\HC$-module category in \cref{prop:momentmapHCmodule}.

To relate our definition of quantum moment maps to Varagnolo--Vasserot's, we also consider \emph{quantum Manin triples} (see \cref{def:quantumManintriple}). Unpacking the categorical definition, a quantum Manin triple consists of a braided monoidal category $\cC$, a pair of monoidal categories $\cD, \cE$ such that $(\cC, \cD)$ and $(\cC, \cE)$ are quantum Manin pairs and a monoidal functor $\cE\otimes_{\cC} \cD\rightarrow \Mod_k$. Here $\cE\otimes_{\cC} \cD$ carries a monoidal structure such that the projection $\cE^{\mop}\otimes \cD\rightarrow \cE\otimes_{\cC} \cD$ is monoidal (see \cref{sect:monoidalmodule} for details). In particular, $\cE\otimes_{\cC} \cD\rightarrow \Mod_k$ gives rise to monoidal functors $F\colon\cD\rightarrow \Mod_k$ and $\tilde{F}\colon \cE^{\mop}\rightarrow \Mod_k$.

The reader may have noticed that we have not included any nondegeneracy assumptions into our definitions of quantum Manin pairs and quantum Manin triples while on the classical level we consider shifted symplectic rather than just shifted Poisson structures. The author is not aware of any nondegeneracy assumptions one may put on quantum Manin pairs which are satisfied in all examples of interest. Furthermore, such nondegeneracy assumptions are not necessary for the applications we consider. In the setting of fusion categories, a closely related definition of quantum Manin pairs and quantum Manin triples was given in \cite[Section 4]{DMNO} and we refer the reader there for possible nondegeneracy assumptions (however, none of the categories we consider are fusion).

We show that a quantum Manin triple encodes a wealth of information: there is an important algebra $\cF=TT^R(1)\in\cD$, a monoidal category $\HC=\cD\otimes_\cC\cD$, a pair of bialgebras $FF^R(k)$ and $\tilde{F}\tilde{F}^R(k)$, an algebra map $F(\cF)\rightarrow \tilde{F}\tilde{F}^R(k)$ and a skew-Hopf pairing
\[\ev\colon \tilde{F}\tilde{F}^R(k)\otimes FF^R(k)\longrightarrow k\]
which allows to turn $FF^R(k)$-comodules into $\tilde{F}\tilde{F}^R(k)$-modules, i.e. it gives a functor
\[\CoMod_{FF^R(k)}\longrightarrow \LMod_{\tilde{F}\tilde{F}^R(k)}.\]
Let us explain this structure in examples:
\begin{itemize}
\item (See \cref{sect:classicalManinTriple}). Given an algebraic group $G$, we have a quantum Manin triple
\[(\CoMod_{\U\g}(\Rep G), \Rep G, \CoMod_{\U\g}).\]
In this case $\cF=\U\g\in\Rep G$, $FF^R(k) = \cO(G)$, $\tilde{F}\tilde{F}^R(k)=\U\g$ and $F(\cF)\rightarrow \tilde{F}\tilde{F}^R(k)$ is an isomorphism and $\ev\colon \U\g\otimes \cO(G)\rightarrow k$ is the obvious pairing. The category $\HC$ is the monoidal category of Harish--Chandra bimodules, i.e. $\U\g$-bimodules where the diagonal action integrates to a $G$-action. Let us recall that the category of Harish--Chandra bimodules has a long history in representation theory: for instance, they are related to blocks in category $\cO$ \cite{BernsteinGelfand} and to character sheaves \cite{BFO}.

\item (See \cref{sect:quantumgroups}). Let $\Rep_q(G)$ be the category of representations of the Lusztig form of the quantum group at an arbitrary quantum parameter $q$. Then
\[(\Rep_q G\otimes \Rep_q (G)^{\bop}, \Rep_q(G), \Rep_q(G^*))\]
is a quantum Manin triple, where $\Rep_q(G^*)$ is the category of comodules over the De Concini--Kac form $\U_q^{DK}(\g)$ of the quantum group. In this case $\cF=\cO_q(G)\in\Rep_q(G)$ is the reflection equation algebra, $FF^R(k) = \cO_q(G)$ and $\tilde{F}\tilde{F}^R(k)=\U_q^{DK}(\g)$. The map $\cF\rightarrow \tilde{F}\tilde{F}^R(k)$ is the Rosso homomorphism $\cO_q(G)\rightarrow \U_q^{DK}(\g)$ (see e.g. \cite[Proposition 10.16]{KlimykSchmudgen}) and the functor $\Rep_q(G)\rightarrow \LMod_{\U_q^{DK}(\g)}$ realizes objects in $\Rep_q(G)$ as modules over the De Concini--Kac quantum group. Note that for $q$ generic this functor is fully faithful. The category $\HC$ is the quantum version of the category of Harish--Chandra bimodules and is equivalent to the Hochschild homology category of $\Rep_q(G)$.
\end{itemize}

Let us state informally some of our results.

\begin{theorem}[\Cref{thm:quantummomentmapManintriple}]
Suppose $(D, H, H^\vee)$ is a triple of Hopf algebras giving rise to a quantum Manin triple $(\CoMod_D, \CoMod_H, \CoMod_{H^\vee})$. Then an algebra map $\mu\colon \cF\rightarrow A$ in $H$-comodules is a quantum moment map iff it satisfies
\[\mu(h)a = (h_{(1)}\triangleright a)\mu(h_{(2)})\]
for every $a\in A$ and $h\in \cF$, where $\Delta(h) = h_{(1)}\otimes h_{(2)}\in H^\vee\otimes \cF$ is the $H^\vee$-coaction on $\cF$.
\end{theorem}

Thus, our quantum moment maps reduce to the quantum moment maps of Varagnolo--Vasserot when we consider quantum Manin triples coming from Hopf algebras.

\begin{theorem}[\Cref{thm:momentmapclassicaldegeneration}]
Suppose $(D, G)$ is a group pair and $(\Rep_\hbar D, \Rep_\hbar G)$ is a quantum Manin pair quantizing it. If an algebra map $\mu_\hbar\colon \cF\rightarrow A_\hbar$ in $\Rep_\hbar G$ is a quantum moment map, then its value $\mu_0\colon \cO(D/G)\rightarrow A_0$ at $\hbar=0$ is a classical moment map.
\end{theorem}

Combining the two results, we conclude that a classical degeneration of Varagnolo--Vasserot moment maps gives Alekseev--Kosmann-Schwarzbach moment maps.

\subsection*{Organization of the paper}

In \cref{sect:background} we recall the necessary facts about monoidal categories that we will use. Since we are interested in categories such as $\Rep G$, the category of all (not necessarily finite-dimensional) representations, we work in the setting of locally presentable categories. It is also a convenient setting for us since the 2-category $\PrL$ of such admits a natural symmetric monoidal structure. In this section we define the notion of a $\cC$-monoidal category, i.e. a monoidal category $\cD$ with a compatible action of a braided monoidal category $\cC$. The pair $(\cC, \cD)$ can be thought of as an algebra in $\PrL$ over the two-dimensional Swiss-cheese operad similar to the description of $\bE_2$-algebras in $\PrL$ in terms of braided monoidal categories. In particular, we show that the relative tensor product $\cE\otimes_\cC \cD$ of $\cC$-monoidal categories $\cE$ and $\cD$ carries a natural monoidal structure and describe its universal property (see \cref{prop:monoidaltensorproduct}).

In \cref{sect:Maninpairs} we give the main definitions of the paper. There we define and study quantum Manin pairs and quantum Manin triples. In particular, we describe the associated algebraic structures, such as analogs of the reflection equation algebra, Rosso homomorphism and the category of Harish--Chandra bimodules. \Cref{sect:coalgebras,sect:classicalManinTriple,sect:quantumgroups} are devoted to examples of quantum Manin triples from Hopf algebras, classical Lie algebras and quantum groups.

Given a quantum Manin pair, we define in \cref{sect:momentmaps} a quantum moment map (see \cref{def:quantummomentmap}) and give several ways to describe them (see \cref{prop:momentmapcentral}). An important observation is that the data of a quantum moment map allows one to extend a $\cD$-module structure to an $\HC$-module structure. We also describe a procedure of fusion of algebras equipped with quantum moment maps. On the level of categories, given two $\HC$-module categories $\cM_1, \cM_2$ it is simply given by the relative tensor product $\cM_1\otimes_\cD \cM_2$.

Finally, in \cref{sect:classicalsetting} we recall definitions of quasi-Poisson groups and quasi-Poisson spaces and provide a definition of moment maps in this setting (\cref{def:classicalmomentmap}). This definition is a slight variant of the definition given in \cite{AKS} and we show that the two are equivalent (see \cref{prop:AKSequivalence}). We also show that for moment maps factoring as $X\rightarrow G^*\rightarrow D/G$, this definition reduces to Lu's definition of the moment map (see \cref{lm:LuAKS}). In \cref{sect:classicallimit} we prove that the classical degeneration of quantum moment maps recovers classical moment maps.

\subsection*{Conventions}

\begin{itemize}
\item We work over the ground commutative ring $k$.

\item $\PrL$ denotes the symmetric monoidal 2-category of $k$-linear locally presentable categories and $k$-linear colimit-preserving functors (see \cite{AdamekRosicky} for a general theory and \cite[Section 2]{BCJF} for a discussion). For $\cC,\cD\in\PrL$ we denote by $\cC\otimes\cD$ the corresponding symmetric monoidal structure. The unit object $\Mod_k\in\PrL$ is the category of $k$-modules; it is compactly generated by finitely presentable $k$-modules.

\item Given two locally presentable categories $\cC,\cD\in\PrL$ we denote by $\FunL(\cC, \cD)\in\PrL$ the category of $k$-linear colimit-preserving functors.
\end{itemize}

\subsection*{Acknowledgements}

The author would like to thank David Jordan for many conversations about quantum groups which in particular inspired the writing of this paper. This research was supported by the NCCR SwissMAP grant of the Swiss National Science Foundation.

\section{Background}

\label{sect:background}

\subsection{Monoidal categories}

We refer to \cite{BJS}, \cite{DSPS}, \cite{EGNO} for a more complete discussion of the notions which we will only briefly recall here.

Let us recall a description of totalizations of cosimplicial categories (i.e. pseudolimits with shape $\Delta$). Suppose $\Delta\rightarrow \PrL$ is a pseudofunctor defining a cosimplicial category $\cC^\bullet$. Denote by $s^i\colon \cC^n\rightarrow \cC^{n-1}$ and $d^i\colon \cC^n\rightarrow \cC^{n+1}$ the codegeneracy and coboundary functors. The limit of $\cC^\bullet$ may be computed by the category of Cartesian sections of the Grothendieck construction (see \cite[Expos\'{e} VI, Chapitre 6.11]{SGA42}) which has the following explicit description.

\begin{lm}
Let $\cC^\bullet$ be a cosimplicial category. Then the limit of $\cC^\bullet$ is equivalent to the following category:
\begin{itemize}
\item Its objects are pairs $(x, \alpha)$ where $x\in\cC^0$ and $\alpha\colon d^1(x)\xrightarrow{\sim} d^0(x)$ such that the diagrams
\[
\xymatrix{
& d^0d^1(x) \ar^{d^0(\alpha)}[r] & d^0d^0(x) & \\
d^2d^0(x) \ar^{\sim}[ur] &&& d^1d^0(x) \ar_{\sim}[ul] \\
& d^2d^1(x) \ar^{d^2(\alpha)}[ul] \ar^{\sim}[r] & d^1d^1(x) \ar_{d^1(\alpha)}[ur]
}
\qquad
\xymatrix{
s^0 d^1(x) \ar^{s^0(\alpha)}[rr] \ar_{\sim}[dr] && s^0 d^0(x) \ar^{\sim}[dl] \\
& x
}
\]
commute.

\item Its morphisms $(x,\alpha)\rightarrow (y, \beta)$ are morphisms $f\colon x\rightarrow y$ such that the diagram
\[
\xymatrix{
d^1(x) \ar^{\alpha}[r] \ar^{d^1(f)}[d] & d^0(x) \ar^{d^0(f)}[d] \\
d^1(y) \ar^{\beta}[r] & d^0(y)
}
\]
commutes.
\end{itemize}
\label{lm:Totofcategories}
\end{lm}

\begin{remark}
From \cref{lm:Totofcategories} we see that we may truncate the cosimplicial object to the first three terms without changing the limit. See also \cite[Theorem 4.11]{Nunes}.
\label{rmk:totalizationtruncation}
\end{remark}

By a monoidal category we will always mean a presentably monoidal category, i.e. a locally presentable category $\cD\in\PrL$ equipped with a monoidal structure whose tensor product functor commutes with colimits in each variable. We denote by $\Alg(\PrL)$ the 2-category of monoidal categories (i.e. pseudo-algebra objects in $\PrL$). Given a monoidal category $\cD$, we denote by $\cD^{\mop}$ the same category equipped with the opposite tensor structure.

Given a monoidal category $\cD$ we denote by
\[\LMod_\cD = \LMod_\cD(\PrL)\]
the 2-category of (left) $\cD$-module categories which are assumed to be locally presentable and such that the action functor $\cD\otimes\cM\rightarrow \cM$ preserves colimits in each variable.

Given a left $\cD$-module category $\cM$ and a right $\cD$-module category $\cN$ we get a simplicial object
\begin{equation}
\simp{\cN\otimes\cM}{\cN\otimes\cD\otimes\cM}
\label{eq:barconstruction}
\end{equation}
in $\PrL$ where the maps come from the action functors on $\cN$ and $\cM$ and the monoidal structure on $\cD$. Note that by a simplicial object in $\PrL_k$ we mean a pseudofunctor $\Delta^{op}\rightarrow \PrL$. Unless the monoidal structure on $\cD$ and the $\cD$-actions on $\cM$ and $\cN$ are strict, this will not be a strict simplicial object.

\begin{defn}
Let $\cM$ be a left $\cD$-module category and $\cN$ a right $\cD$-module category. Their \defterm{relative tensor product} is the colimit
\begin{equation}
\cN\otimes_\cD\cM = \colim\left(\simp{\cN\otimes\cM}{\cN\otimes\cD\otimes\cM}\right)
\label{eq:relativetensorproduct}
\end{equation}
in $\PrL$.
\end{defn}

\begin{remark}
We may compute the pseudo-colimit in $\PrL$ as a homotopy colimit in the canonical model structure on $\PrL$, see \cite{Gambino}. Therefore, if $\cPrL_1$ denotes the $\infty$-category obtained by applying the Duskin nerve to the underlying $(2, 1)$-category of $\PrL$, the pseudo-colimit may also be computed as the $\infty$-categorical colimit in $\cPrL_1$.
\end{remark}

\begin{remark}
Since all functors in \eqref{eq:barconstruction} admit a right adjoint, we may compute the colimit in \eqref{eq:relativetensorproduct} as a limit of the cosimplicial diagram obtained by passing to right adjoints. In particular, from \cref{rmk:totalizationtruncation} we see that we may truncate the diagram \eqref{eq:barconstruction} to the first three terms without changing the colimit.
\end{remark}

The relative tensor product $\cN\otimes_\cD\cM$ satisfies the following universal property. Let $\cA\in\PrL$ be a category. Recall the following notion (see \cite[Definition 3.1]{ENO}, \cite[Definition 3.1]{DSPS}).

\begin{defn}
A \defterm{$\cD$-balanced functor} $F\colon \cN\times\cM\rightarrow \cA$ is a bifunctor preserving colimits in each variable equipped with an isomorphism
\[\alpha_{W, X, V}\colon F(W\otimes X, V)\cong F(W, X\otimes V)\]
natural in $X\in\cD, W\in\cN, V\in\cM$ which makes the obvious diagrams commute.
\end{defn}

\begin{prop}
Let $\cA\in\PrL$ be a locally presentable category. Then $\FunL(\cN\otimes_\cD\cM, \cA)$ is equivalent to the category of $\cD$-balanced functors $\cN\otimes\cM\rightarrow \cA$.
\label{prop:balancedfunctors}
\end{prop}
\begin{proof}
The category $\FunL(\cN\otimes_\cD\cM, \cA)$ is equivalent to the limit
\[\lim\left(\cosimp{\FunL(\cN\otimes\cM, \cA)}{\FunL(\cN\otimes\cD\otimes\cM, \cA)}\right)\]
of the cosimplicial object in $\PrL$. So, by \cref{lm:Totofcategories} we get that an object in $\FunL(\cN\otimes_\cD \cM, \cA)$ is given by a colimit-preserving functor $F\colon \cN\otimes\cM\rightarrow \cA$ together with a natural transformation $\alpha$ as above which satisfies a pair of coherence relations. This is exactly the description of $\cD$-balanced functors $\cN\otimes\cM\rightarrow \cA$.
\end{proof}

\begin{defn}
Let $\cC$ be a monoidal category. An object $A\in\cC$ is \defterm{faithfully flat} if the functor $A\otimes -$ is conservative and preserves equalizers.
\end{defn}

We will use the following result to work with the category of comodules over a Hopf algebra (it was previously proved in \cite[Example 4.8]{BruguieresVirelizier} under similar assumptions).

\begin{thm}[Fundamental theorem of Hopf modules]
Let $\cC$ be a braided monoidal category (not necessarily locally presentable) which admits equalizers and $H\in\cC$ a faithfully flat Hopf algebra in $\cC$. Then there is an equivalence
\[\cC\xrightarrow{\sim} \CoMod_H(\LMod_H(\cC))\]
given by $V\mapsto H\otimes V$.
\label{thm:hopfmodules}
\end{thm}
\begin{proof}
Consider the functor $F\colon \cC\rightarrow \LMod_H(\cC)$ given by $F(V) = H\otimes V$ (the free left $H$-module). It admits a right adjoint $G\colon \LMod_H(\cC)\rightarrow \cC$ given by the forgetful functor.

The functor $G$ is conservative and preserves limits. Moreover, by assumption $GF$ is conservative and preserves equalizers. Therefore, $F$ is conservative and preserves equalizers. By the Barr--Beck theorem \cite[Theorem VI.7.1]{Maclane} we conclude that $F$ is comonadic. The comonad $T = FG$ on $\LMod_H(\cC)$ sends $V$ to $H\otimes V_{\triv}$, where $V_{\triv}$ denotes the trivial $H$-module structure (i.e. $H$ acts via the counit).

We have another comonad $T'$ on $\LMod_H(\cC)$ given by $T'(V) = H\otimes V$. We have a natural morphism
\[\alpha_V\colon H\otimes V_{\triv}\longrightarrow H\otimes V\]
given by
\[\alpha_V(h\otimes v) = h_{(1)}\otimes h_{(2)} v\]
and it is easy to see that it is compatible with the comonad structures on $T$ and $T'$. $\alpha$ has an inverse
\[\beta_V(h\otimes v) = h_{(1)}\otimes S(h_{(2)}) v\]
and hence
\[\cC\cong \CoAlg_{T'}(\LMod_H(\cC)) = \CoMod_H(\LMod_H(\cC))\]
\end{proof}

\begin{remark}
Note that if $\cC$ is abelian and the tensor product preserves direct sums, for any Hopf algebra $H\in\cC$ the functor $H\otimes -$ is conservative since $H\cong 1\oplus \ker(\epsilon)$.
\end{remark}

We will also often use the following statement (see \cite[Corollary 4.13]{BZBJ1}) which is an application of the Barr--Beck theorem \cite[Theorem VI.7.1]{Maclane}.

\begin{prop}
Let $\cD$ be a monoidal category, $\cM$ a $\cD$-module category and $A\in\cD$ an algebra object. Then the functor
\[\LMod_A(\cD)\otimes_\cD \cM\longrightarrow \LMod_A(\cM)\]
given by $M\boxtimes V\mapsto M\otimes V$ is an equivalence.
\label{prop:modtensorproduct}
\end{prop}

\subsection{Monoidal module categories}
\label{sect:monoidalmodule}

A monoidal category $\cD$ is naturally a left and right module over itself, so it defines an object $\cD\in\LMod_{\cD\otimes\cD^{\mop}}$.

\begin{defn}
Let $\cD\in\PrL_k$ be a monoidal category. Its \defterm{Drinfeld center} is the category
\[\Z(\cD) = \Hom_{\LMod_{\cD\otimes \cD^{\mop}}}(\cD, \cD).\]
\end{defn}

Explicitly, $\Z(\cD)$ has objects given by pairs $(x, \beta)$ of an object $x\in\cD$ and a natural isomorphism $\beta\colon x\otimes(-)\xrightarrow{\sim} (-)\otimes x$. It is naturally a braided monoidal category.

For a braided monoidal category $\cC$ we denote by $\cC^{\bop}$ the same monoidal category with the inverse braiding. Then for a monoidal category $\cD$ we have a natural braided monoidal equivalence
\[\Z(\cD^{\mop})\cong \Z(\cD)^{\bop}.\]

\begin{defn}
Let $\cC$ be a braided monoidal category. A \defterm{$\cC$-monoidal category} is a monoidal category $\cD\in\PrL_k$ together with a braided monoidal functor $\cC\rightarrow \Z(\cD)$.
\end{defn}

\begin{remark}
The same notion was previously called a \emph{tensor category over $\cC$} in \cite[Definition 4.16]{DGNO} and a \emph{$\cC$-algebra} in \cite[Definition 3.2]{BJS}.
\end{remark}

Explicitly (see \cite[Definition 2.1]{Bez}) we have a monoidal action functor $T\colon \cC\rightarrow \cD$ and isomorphisms
\[\tau_{X, V}\colon T(X)\otimes V\xrightarrow{\sim} V\otimes T(X)\]
natural in $X\in\cC$ and $V\in\cD$ which make obvious diagrams commute. 

\begin{example}
If $\cC$ is a braided monoidal category, it can also be considered as a $\cC$-monoidal category.
\end{example}

\begin{example}
Suppose $A\in\cC$ is a commutative algebra. Then the category $\LMod_A(\cC)$ of (left) $A$-modules in $\cC$ becomes a $\cC$-monoidal category with the functor $\cC\rightarrow \LMod_A(\cC)$ given by $X\mapsto A\otimes X$ and the isomorphism $\tau$ given by the braiding.
\label{ex:modulescommutativealgebra}
\end{example}

\begin{example}
Suppose $A\in\cC$ is a bialgebra. Then the category $\CoMod_A(\cC)$ of $A$-comodules in $\cC$ becomes a $\cC$-monoidal category with the functor $\cC\rightarrow \CoMod_A(\cC)$ given by the trivial module and the isomorphism $\tau$ the braiding on $\cC$.
\label{ex:comodulesbialgebra}
\end{example}

\begin{remark}
Using the braided monoidal equivalence $\Z(\cD^{\mop})\cong \Z(\cD)^{\bop}$, a $\cC$-monoidal category gives rise to a $\cC^{\bop}$-monoidal category.
\end{remark}

Since $T\colon \cC\rightarrow \cD$ is monoidal and continuous, there is a right adjoint $T^R\colon \cD\rightarrow \cC$ which is moreover lax monoidal.

\begin{prop}
The right adjoint $T^R\colon \cD\rightarrow \cC$ is a lax $\cC$-monoidal functor, i.e. the diagram
\[
\xymatrix@C=2cm{
X\otimes T^R(V) \ar^{\sigma_{X, T^R(V)}}[r] \ar[d] & T^R(V)\otimes X \ar[d] \\
T^RT(X)\otimes T^R(V) \ar^{\sigma_{T^R T(X), T^R(V)}}[r] \ar[d] & T^R(V)\otimes T^R T(X) \ar[d] \\
T^R(T(X)\otimes V) \ar^{T^R(\tau_{X, V})}[r] & T^R(V\otimes T(X))
}
\]
commutes.
\label{prop:rightadjointlaxmonoidal}
\end{prop}

The proof of the above theorem is identical to the proof that the right adjoint of a braided monoidal functor preserves the braiding, so we omit it.

\begin{lm}
Let $\cC$ be a braided monoidal category and $\cD$ a $\cC$-monoidal category. Then:
\begin{enumerate}
\item The tensor product functor $\cC\otimes \cC\rightarrow \cC$ carries a natural monoidal structure.

\item The action functors $\cC\otimes \cD\rightarrow \cD$ given by $X,V\mapsto T(X)\otimes V$ and $\cD^{\mop}\otimes \cC\rightarrow \cD^{\mop}$ given by $V, X\mapsto V\otimes T(X)$ carry a natural monoidal structure.
\end{enumerate}
\end{lm}
\begin{proof} $ $
\begin{enumerate}
\item The monoidal structure on $\cC\otimes \cC\rightarrow \cC$ is the natural isomorphism
\[(X_1\otimes X_2)\otimes (X_3\otimes X_4)\cong (X_1\otimes X_3)\otimes (X_2\otimes X_4)\]
given by the braiding $\sigma_{X_2, X_3}$.

\item Let $T\colon \cC\rightarrow \cD$ be the monoidal action functor. The monoidal structure on $\cC\otimes \cD\rightarrow \cD$ is the natural isomorphism
\[(T(X_1)\otimes V_1)\otimes (T(X_2)\otimes V_2)\cong T(X_1\otimes X_2)\otimes (V_1\otimes V_2)\]
given by $\tau_{X_2, V_1}^{-1}$ and the monoidal structure on $T$.
\end{enumerate}
\end{proof}

Given a $\cC$-monoidal category $\cD$ and a $\cC$-monoidal category $\cE$, we can therefore upgrade the simplicial object \eqref{eq:barconstruction} in $\PrL$ to a simplicial object
\[\simp{\cE^{\mop}\otimes \cD}{\cE^{\mop}\otimes\cC\otimes\cD}\]
in $\Alg(\PrL)$. By \cite[Proposition 3.2.3.1]{HA} the forgetful functor $\Alg(\PrL)\rightarrow \PrL$ creates geometric realizations of simplicial objects, so the relative tensor product $\cE\otimes_\cC\cD$ carries a natural structure of a monoidal category. We are now going to explain a universal property of the monoidal structure on the relative tensor product $\cE\otimes_\cC\cD$.

\begin{defn}
Let $\cD,\cE,\cA$ be monoidal categories and $F_\cE\colon \cE\rightarrow \cA$ and $F_\cD\colon \cD\rightarrow \cA$ be monoidal functors. Then a \defterm{distributive law} between $F_\cE$ and $F_\cD$ is given by an isomorphism
\[\beta_{W, V}\colon F_\cE(W)\otimes F_\cD(V)\xrightarrow{\sim} F_\cD(V)\otimes F_\cE(W)\]
compatible with the monoidal structures on $F_\cE$ and $F_\cD$.
\end{defn}

Let $\Fun^{\otimes}(-, -)$ denote the category of colimit-preserving monoidal functors of monoidal categories. The following is then well-known.

\begin{lm}
Suppose $\cE, \cD, \cA$ are monoidal categories. Then $\Fun^{\otimes}(\cE\otimes \cD, \cA)$ is equivalent to the category of triples $(F_\cE, F_\cD, \beta)$ of monoidal functors $F_\cE\colon \cE\rightarrow \cA$ and $F_\cD\colon \cD\rightarrow \cA$ and a distributive law between them.
\label{lm:distributivelaw}
\end{lm}

One also has a similar description of the functor category for several tensor factors. Let $T_\cE\colon \cC\rightarrow \cE^{\mop}$ and $T_\cD\colon\cC\rightarrow \cD$ the monoidal action functors.

\begin{prop}
Suppose $\cC$ is a braided monoidal category, $\cD$ and $\cE$ are $\cC$-monoidal categories and $\cA$ is another monoidal category. Then $\Fun^{\otimes}(\cE\otimes_\cC \cD, \cA)$ is equivalent to the following category:
\begin{itemize}
\item Its objects are quadruples $(F_\cE, F_\cD, \beta, \alpha)$, where $F_\cE\colon \cE^{\mop}\rightarrow \cA$ and $F_\cD\colon \cD\rightarrow \cA$ are monoidal functors, $\beta_{W, V}\colon F_\cE(W)\otimes F_\cD(V)\xrightarrow{\sim} F_\cD(V)\otimes F_\cE(W)$ is a distributive law and $\alpha\colon F_\cD\circ T_\cD\xrightarrow{\sim} F_\cE\circ T_\cE$ is a monoidal natural isomorphism such that the diagram
\[
\xymatrix{
F_\cD(T_\cD(X)\otimes V) \ar^{F_\cD(\tau_{X, V})}[r] \ar^{\sim}[d] & F_\cD(V\otimes T_\cD(X)) \ar^{\sim}[d] \\
F_\cD(T_\cD(X))\otimes F_\cD(V) \ar^{\alpha_X\otimes \id}[d] & F_\cD(V)\otimes F_\cD(T_\cD(X)) \ar^{\id\otimes\alpha_X}[d] \\
F_\cE(T_\cE(X))\otimes F_\cD(V) \ar^{\beta_{T_\cE(X), V}}[r] & F_\cD(V)\otimes F_\cE(T_\cE(X))
}
\]
and its analog for $\cE$ commute.

\item Its morphisms $(F_\cE, F_\cD, \beta, \alpha)\rightarrow (F_\cE', F_\cD', \beta', \alpha')$ are monoidal natural transformations $F_\cE\rightarrow F_\cE'$ and $F_\cD\rightarrow F_\cD'$ compatible with the isomorphisms $\alpha$ and $\beta$.
\end{itemize}
\label{prop:monoidaltensorproduct}
\end{prop}
\begin{proof}
The proof is analogous to the proof of \cref{prop:balancedfunctors} where we use \cref{lm:distributivelaw} to describe the categories $\Fun^{\otimes}(\cE^{\mop}\otimes \cD, \cA)$ and $\Fun^{\otimes}(\cE^{\mop}\otimes \cC\otimes \cD, \cA)$.
\end{proof}

\section{Manin pairs}

\label{sect:Maninpairs}

Our goal in this section is to provide a definition and examples of quantization of Manin pairs and Manin triples.

\subsection{Quantum Manin pairs}

\begin{defn}
A \defterm{quantum Manin pair} $(\cC, \cD)$ is a pair of a braided monoidal category $\cC$ and a $\cC$-monoidal category $\cD$.
\label{def:quantumManinpair}
\end{defn}

\begin{example}
A somewhat degenerate example of a quantum Manin pair is the pair $(\Mod_k, \cD)$ for a monoidal category $\cD$. Namely, $k$-linear colimit-preserving functors $\Mod_k\rightarrow \Z(\cD)$ by the Eilenberg--Watts theorem \cite{NymanSmith} are uniquely specified by their value on $k\in\Mod_k$ and we send $k\in\Mod_k$ to the unit object of $\Z(\cD)$.
\end{example}

\begin{example}
If $\cD$ is a braided monoidal category, the pair $(\cD\otimes \cD^{\bop}, \cD)$ is a quantum Manin pair via the natural braided monoidal functor $\cD\otimes \cD^{\bop}\rightarrow \Z(\cD)$ (see \cite[Proposition 8.6.1]{EGNO}) corresponding to the left and right action of $\cD$ on itself.
\label{ex:factorizableManinpair}
\end{example}

\begin{remark}
Suppose $(\fd, \g)$ is a Manin pair which integrates to a group pair $(D, G)$ \cite{AKS}. As shown in \cite[Proposition 4.16]{PoissonLie}, in this case the map of classifying stacks $\B G\rightarrow \B D$ carries a 2-shifted Lagrangian structure and the quantum Manin pair $(\cC, \cD)$ may be thought of as a quantization of this Lagrangian.
\label{rmk:classicalmaninpairs}
\end{remark}

We are now going to define several important objects associated to a quantum Manin pair $(\cC, \cD)$. By the results of \cref{sect:monoidalmodule} the category
\[\HC = \cD\otimes_{\cC} \cD\]
carries a natural structure of a monoidal category such that the projection
\[\cD^{\mop}\otimes \cD\longrightarrow \HC\]
is monoidal.

\begin{example}
Suppose $\cD$ is a braided monoidal category and consider the quantum Manin pair $(\cD\otimes \cD^{\bop}, \cD)$. Then the monoidal category
\[\HC = \cD\otimes_{\cD\otimes \cD^{\bop}} \cD\]
is a twisted version of the cocenter (or zeroth Hochschild homology) category of $\cD$, see \cite[Lemma 3.9]{BZBJ2}. If $\cD$ is balanced, the two coincide.
\end{example}

\begin{remark}
In the setting of \cref{rmk:classicalmaninpairs} the monoidal category $\HC$ may be thought of as a quantization of the 1-shifted symplectic stack $\B G\times_{\B D} \B G\cong [G\backslash D/G]$.
\end{remark}

We have the following description of $\HC$-monoidal categories.

\begin{prop}
An $\HC$-module category is a $\cD$-bimodule category $\cM$ together with an identification of the two induced $\cC$-actions such that the diagram
\[
\xymatrix@C=2cm{
(T(X)\otimes V)\otimes M \ar^{\tau_{X, V}\otimes \id}[r] \ar^{\sim}[d] & (V\otimes T(X))\otimes M \ar^{\sim}[d] \\
T(X)\otimes(V\otimes M) \ar^{\sim}[d] & V\otimes (T(X)\otimes M) \ar^{\sim}[d] \\
(V\otimes M)\otimes T(X) \ar^{\sim}[r] & V\otimes (M\otimes T(X))
}
\]
for $X\in\cC, V\in\cD$ and $M\in\cM$ commutes and similarly for the right action.
\label{prop:HCmodule}
\end{prop}
\begin{proof}
An $\HC$-module category is a category $\cM$ together with a monoidal functor
\[\cD\otimes_\cC\cD\rightarrow \FunL(\cM, \cM).\]
The latter can be unpacked using \cref{prop:monoidaltensorproduct}.
\end{proof}

\begin{example}
$\cD$ is a $\cD$-bimodule category with respect to the left and right actions on itself. The identification of the induced $\cC$-actions is given by the isomorphism $\tau$, so $\cD$ is an $\HC$-module category.
\label{ex:trivialHCmodule}
\end{example}

\begin{remark}
Consider the quantum Manin pair $(\cD\otimes \cD^{\bop}, \cD)$ from \cref{ex:factorizableManinpair} and an $\HC$-module category $\cM$. Then we get two identifications of the left and right $\cD$-actions. Therefore, in this case an $\HC$-module category is a $\cD$-module category $\cM$ together with an automorphism of the action functor $\cD\otimes\cM\rightarrow \cM$ satisfying certain compatibilities. In other words, $\cM$ is a $\cD$-braided module category, see \cite[Definition 3.4]{BZBJ2} and \cite[Theorem 3.11]{BZBJ2}.
\end{remark}

We may introduce a relative version of the Drinfeld center, see also \cite[Definition 3.28]{Laugwitz} for a related notion.

\begin{defn}
Let $\cD$ be a $\cC$-monoidal category. The \defterm{relative Drinfeld center} is
\[\Z_{\cC}(\cD) = \FunL_{\HC}(\cD, \cD),\]
the category of colimit-preserving functors $\cD\rightarrow \cD$ as $\HC$-module categories.
\end{defn}

Explicitly (see \cite[Proposition 3.34]{Laugwitz}), $\Z_{\cC}(\cD)$ is a full subcategory of the Drinfeld center $\Z(\cD)$ consisting of pairs $(V, \beta)$ such that $\beta_{T(x)}\colon V\otimes T(x)\xrightarrow{\sim} T(x)\otimes V$ is the inverse of $\tau_{x, V}$.

Let $T\colon \cC\rightarrow \cD$ be the monoidal action functor. Its right adjoint $T^R\colon \cD\rightarrow \cC$ is lax monoidal and so $T^R(1)\in\cC$ is an algebra object.

\begin{prop}
The algebra $T^R(1)\in\cC$ is commutative.
\label{prop:REAcommutative}
\end{prop}
\begin{proof}
Consider the diagram
\[
\xymatrix{
T^R(1) \otimes T^R(1) \ar^{\sigma_{T^R(1), T^R(1)}}[rr] \ar[d] && T^R(1)\otimes T^R(1) \ar[d] \\
T^R(TT^R(1)\otimes 1) \ar^{\tau_{T^R(1), 1}}[rr] \ar[dr] && T^R(1\otimes TT^R(1)) \ar[dl] \\
& T^R(1)
}
\]
The top square commutes by \cref{prop:rightadjointlaxmonoidal} and the bottom square commutes since $\tau_{-, 1}$ is the unit isomorphism. The two composites $T^R(1)\otimes T^R(1)\rightarrow T^R(1)$ coincide with the multiplication $T^R(1)\otimes T^R(1)\rightarrow T^R(1)$ and so the diagram expresses commutativity of the algebra $T^R(1)$.
\end{proof}

\begin{remark}
There is a related notion of quantum Manin pairs $(\cC, A)$ introduced in \cite[Definition 4.2]{DMNO} which are given by a non-degenerate braided fusion category $\cC$ and a commutative algebra $A$ in $\cC$ satisfying certain assumptions. Then the pair $(\cC, \LMod_A(\cC))$ is a quantum Manin pair in our sense.
\end{remark}

Define
\[\cF=TT^R(1)\]
which is an algebra in $\cD$. The counit of the adjunction $T\dashv T^R$ gives rise to an algebra map
\begin{equation}
\epsilon\colon \cF\rightarrow 1
\label{eq:epsilon}
\end{equation}
in $\cD$.

\begin{remark}
Consider a group pair $(D, G)$. Recall from \cite[Section 3.5]{AKS} that $D/G$ is a quasi-Poisson $D$-space. Then we may think of $T^R(1)$ as a quantization of $\cO(D/G)\in\Rep D$.
\end{remark}

\subsection{Monadic case}

Consider a quantum Manin pair $(\cC, \cD)$.

\begin{defn}
The action functor $T\colon \cC\rightarrow \cD$ \defterm{satisfies the projection formula} if the natural morphism
\[X\otimes T^R(V)\longrightarrow T^R(T(X)\otimes V)\]
is an isomorphism for every $X\in\cC$ and $V\in\cD$.
\end{defn}

\begin{remark}
By \cref{prop:rightadjointlaxmonoidal} $T$ satisfies the projection formula iff
\[T^R(V)\otimes X\rightarrow T^R(V\otimes T(X))\]
is an isomorphism.
\end{remark}

Suppose $T$ satisfies the projection formula. Note that since $T^R(1)$ is a commutative algebra, $\LMod_{T^R(1)}(\cC)$ is a $\cC$-monoidal category by \cref{ex:modulescommutativealgebra}.

\begin{prop}
The natural functor $\cD\rightarrow \LMod_{T^R(1)}(\cC)$ given by $V\mapsto T^R(V)$ is a functor of $\cC$-monoidal categories.
\end{prop}
\begin{proof}
The projection formula implies that $\cD\rightarrow \LMod_{T^R(1)}(\cC)$ is a functor of $\cC$-module categories. Let us now show that it is strictly compatible with the monoidal structures. The natural lax monoidal structure on $T^R\colon \cD\rightarrow \cC$ gives a fork
\[
\xymatrix{
T^R(V)\otimes T^R(1)\otimes T^R(W) \ar@<.5ex>[r] \ar@<-.5ex>[r] & T^R(V)\otimes T^R(W) \ar[r] & T^R(V\otimes W)
}
\]
in $\cC$ which we have to prove is a coequalizer which will give the required monoidal structure
\[T^R(V)\otimes_{T^R(1)} T^R(W)\cong T^R(V\otimes W).\]

Using the projection formula we may identify the above fork with $T^R(-\otimes W)$ applied to the fork
\[
\xymatrix{
TT^R(V)\otimes TT^R(1) \ar@<.5ex>[r] \ar@<-.5ex>[r] & TT^R(V) \ar[r] & V
}
\]
in $\cD$.

Using the monoidal structure on $T$ and the projection formula again, we identify the above fork with
\[
\xymatrix{
TT^RTT^R(V) \ar@<.5ex>[r] \ar@<-.5ex>[r] & TT^R(V) \ar[r] & V
}
\]
which is a split coequalizer.
\end{proof}

\begin{remark}
Consider the quantum Manin pair from \cref{ex:factorizableManinpair}. If $\cD$ is rigid, i.e. every compact object is dualizable, then \cite[Proposition 3.17]{BZBJ1} implies that $T\colon \cD\otimes \cD\rightarrow \cD$ satisfies the projection formula and $T^R\colon \cD\rightarrow \cD\otimes \cD$ is monadic.
\end{remark}

We have a natural isomorphism
\[\tau_{T^R(1), V}\colon \cF\otimes V\xrightarrow{\sim} V\otimes \cF\]
for any $V\in\cD$ given by the $\cC$-monoidal structure on $\cD$ called the \emph{field goal isomorphism} in \cite[Corollary 4.6]{BZBJ2}. In particular, this induces a monoidal structure on $\LMod_\cF(\cD)$ given by turning a left $\cF$-module into a right $\cF$-module using the field goal isomorphism and applying the relative tensor product of modules. 

Recall $T^R\colon\cD\rightarrow \cC$ is monadic if the natural functor
\[\cD\longrightarrow \LMod_{T^R(1)}(\cC)\]
is an equivalence.

\begin{prop}
Suppose $T\colon \cC\rightarrow \cD$ satisfies the projection formula and $T^R\colon \cD\rightarrow \cC$ is monadic. Then there is a monoidal equivalence
\[\HC\cong \LMod_\cF(\cD)\]
where the forgetful functor $\HC\rightarrow \cD$ is right adjoint to the monoidal functor $\cD\rightarrow \HC$ given by $V\mapsto 1\boxtimes V$.
\label{prop:HCmonadic}
\end{prop}
\begin{proof}
The $\cC$-monoidal equivalence $\cD\cong \LMod_{T^R(1)}(\cC)$ gives rise to a monoidal equivalence
\[\HC\cong \LMod_{T^R(1)}(\cC)\otimes_\cC \cD.\]
We have a natural monoidal functor
\[\LMod_{T^R(1)}(\cC)\otimes_\cC\cD\longrightarrow \LMod_\cF(\cD)\]
given by $M\boxtimes V\mapsto T(V)\otimes V$ which by \cref{prop:modtensorproduct} is an equivalence.
\end{proof}

\begin{example}
In the case of the quantum Manin pair $(\cD\otimes \cD^{\bop}, \cD)$ we get the monoidal equivalence $\HC\cong \LMod_\cF(\cD)$ from \cite[Section 4.2]{BZBJ2}.
\end{example}

\subsection{Quantum Manin triples}

\begin{defn}
A \defterm{quantum Manin triple} is a triple $(\cC, \cD, \cE)$ where $(\cC, \cD)$ and $(\cC, \cE)$ are quantum Manin pairs together with a monoidal functor
\[\cE\otimes_\cC \cD\longrightarrow \Mod_k.\]
\label{def:quantumManintriple}
\end{defn}

\begin{remark}
There is a related notion of quantum Manin triples $(\cC, A)$ introduced in \cite[Definition 4.13]{DMNO} which are given by a non-degenerate braided fusion category $\cC$ and a pair of commutative algebras $A,B$ in $\cC$ such that the category of $(A,B)$-bimodules in $\cC$ is equivalent to $\Mod_k$ and which satisfy some extra assumptions. We have $\cC$-monoidal categories $\cD = \LMod_A(\cC)$ and $\cE=\LMod_B(\cC)$. Moreover, the category $\cE\otimes_{\cC}\cD$ is equivalent to the category of $(A,B)$-bimodules in $\cC$ and hence $(\cC, \LMod_A(\cC), \LMod_B(\cC))$ gives a quantum Manin triple in our sense.
\end{remark}

\begin{remark}
Suppose $(\fd, \g, \g^*)$ is a Manin triple which integrates to a group triple $(D, G, G^*)$. As shown in \cite[Proposition 4.17]{PoissonLie}, in this case the maps of classifying stacks
\[
\xymatrix{
& \pt \ar[dl] \ar[dr] & \\
\B G^* \ar[dr] && \B G \ar[dl] \\
& \B D &
}
\]
form an iterated Lagrangian correspondence. The notion of a quantum Manin triple may be thought of as a quantization of this Lagrangian correspondence.
\end{remark}

Let $T\colon \cC\rightarrow \cD$ and $\tilde{T}\colon \cC\rightarrow \cE$ be the action functors. Using \cref{prop:monoidaltensorproduct} we may unpack the monoidal functor $\cE\otimes_\cC\cD\rightarrow \Mod_k$ in terms of the following data:
\begin{itemize}
\item A pair of monoidal functors $F\colon \cD\rightarrow \Mod_k$ and $\tilde{F}\colon \cE^{\mop}\rightarrow \Mod_k$.

\item A distributive law $\beta_{V, W}\colon F(V)\otimes \tilde{F}(W)\xrightarrow{\sim} \tilde{F}(W)\otimes F(V)$.

\item A monoidal isomorphism $\alpha\colon F\circ T\xrightarrow{\sim} \tilde{F}\circ \tilde{T}$.
\end{itemize}

In particular, we get two algebras $F^R(k)\in\cD$ and $\tilde{F}^R(k)\in\cE$ so that their images $FF^R(k), \tilde{F}\tilde{F}^R(k)\in\Mod_k$ are bialgebras, where the coproducts come from the comonad structures on $FF^R$ and $\tilde{F}\tilde{F}^R$ (see \cite[Section 5.4]{EGNO}). The algebra map $\epsilon\colon \cF\rightarrow 1$ gives rise to an algebra map $F(\epsilon)\colon \tilde{F}\tilde{T} T^R(1)\cong F(\cF)\rightarrow k$ which by adjunction gives an algebra map
\begin{equation}
\tilde{T} T^R(1)\longrightarrow \tilde{F}^R(k).
\label{eq:REAFRT}
\end{equation}

\begin{defn}
A \defterm{factorizable quantum Manin triple} is a quantum Manin triple $(\cD\otimes \cD^{\bop}, \cD, \cE)$, where $(\cD\otimes \cD^{\bop}, \cD)$ is the quantum Manin pair from \cref{ex:factorizableManinpair}.
\end{defn}

\begin{prop}
Suppose $T\colon \cD\otimes \cD\rightarrow \cD$ satisfies the projection formula and $T^R\colon \cD\rightarrow \cD\otimes \cD$ preserves colimits. Then we have an isomorphism of algebras
\[F(\cF)\cong FF^R(k).\]
\label{prop:factorizableFRTREA}
\end{prop}
\begin{proof}
The functor $T\colon \cD\otimes \cD\rightarrow \cD$ is a 1-morphism in $\LMod_{\cD\otimes \cD}$ and by assumptions it admits a right adjoint $T^R\colon \cD\rightarrow \cD\otimes \cD$ in the same 2-category. We have a monoidal functor $F\otimes \id\colon \cD\otimes \cD\rightarrow \cD$ which gives rise to a 2-functor
\[\LMod_{\cD\otimes \cD}\longrightarrow \LMod_\cD\]
under which the image of $T$ is $F$. Therefore, the image of $T^R$ under the above functor is $F^R$. The commutative diagram of categories
\[
\xymatrix{
\cD \ar^{T^R}[r] \ar[d] & \cD\otimes \cD \ar[d] \\
\cD\otimes_\cD\Mod_k \ar^-{T^R\otimes \id}[r] & (\cD\otimes\cD)\otimes_\cD\Mod_k
}
\]
then implies that the natural morphism
\[(F\otimes \id)T^R\rightarrow F^R F\]
is an isomorphism. Therefore,
\[(F\otimes F)T^R\rightarrow F F^R F\]
is an isomorphism and the claim follows.
\end{proof}

\begin{example}
As shown in \cite[Propositions 3.11, 3.12]{BZBJ1}, the assumptions of the proposition are satisfied when $\cD$ is rigid, i.e. every compact object is dualizable.
\end{example}

\subsection{Coalgebras}
\label{sect:coalgebras}

In this section we describe quantum Manin triples arising from a triple of bialgebras. Recall the notion of a skew-pairing $\gamma\colon H^\vee\otimes H\rightarrow k$ of Hopf algebras, see e.g. \cite[Definition 8.3]{KlimykSchmudgen}. We denote by $\gamma^{-1}\colon H^\vee\otimes H\rightarrow k$ the convolution inverse and by $\overline{\gamma}\colon H\otimes H^\vee\rightarrow k$ the inverse skew-pairing which is $\gamma^{-1}$ precomposed with the tensor flip. Then we have the following notion, see e.g. \cite[Definition 10.1]{KlimykSchmudgen}.

\begin{defn}
Let $D$ be a Hopf algebra. A \defterm{coquasitriangular structure} on $D$ is a skew-pairing $\br_D\colon D\otimes D\rightarrow k$ such that
\begin{equation}
\br_D(a_{(1)}, b_{(1)})a_{(2)}b_{(2)} = \br_D(a_{(2)}, b_{(2)})b_{(1)}a_{(1)}.
\label{eq:Rmatrix}
\end{equation}
\end{defn}

For a Hopf algebra $D$ the category $\CoMod_D$ of right $D$-comodules is locally presentable \cite[Corollary 26]{Wischnewsky} and hence is a monoidal category in our sense. Moreover, if $D$ is a coquasitriangular Hopf algebra, $\CoMod_D$ is a braided monoidal category, where the braiding $V\otimes W\xrightarrow{\sim} W\otimes V$ is defined by
\[v\otimes w\mapsto \br_D(v_{(1)}, w_{(1)}) w_{(0)}\otimes v_{(0)}\]
The opposite braided monoidal structure on $\CoMod_D$ is defined using $\overline{\br}$.

\begin{defn}
Let $D$ be a coquasitriangular Hopf algebra and $f\colon D\rightarrow H$ a morphism of Hopf algebras. A \defterm{$D$-coquasitriangular structure} on $H$ is a skew-pairing $\br_H\colon D\otimes H\rightarrow k$ such that
\begin{enumerate}
\item $\br_H(d_1, f(d_2)) = \br_D(d_1, d_2)$ for all $d_1, d_2\in D$.

\item $\br_H(d_{(1)}, h_{(1)}) f(d_{(2)})h_{(2)} = \br_H(d_{(2)}, h_{(2)})h_{(1)}f(d_{(1)})$ for all $d\in D$ and $h\in H$.
\end{enumerate}
\end{defn}

\begin{remark}
If $f\colon D\rightarrow H$ is surjective, a $D$-coquasitriangular structure on $H$ exists iff $\br_D(d_1, d_2) = 0$ for every $d_1, d_2\in D$ such that $f(d_2) = 0$ in which case it is unique.
\end{remark}

Given a $D$-coquasitriangular structure on a Hopf algebra $H$, we get a $\CoMod_D$-monoidal structure on $\CoMod_H$. Now fix the following data:
\begin{itemize}
\item A coquasitriangular Hopf algebra $D$.

\item A $(D, \br_D)$-coquasitriangular Hopf algebra $H$ with a Hopf map $f\colon D\rightarrow H$.

\item A $(D, \overline{\br}_D)$-coquasitriangular Hopf algebra $H^\vee$ with a Hopf map $g\colon D\rightarrow H^\vee$.

\item A skew-pairing $\ev\colon H^\vee\otimes H\rightarrow k$ such that
\[\br_H(d, h) = \ev(g(d), h),\qquad \br_{H^\vee}^{-1}(d, h) = \ev(h, f(d)).\]
\end{itemize}

\begin{remark}
The compatibility between the skew-pairings implies that the data boils down to a triple of Hopf algebras $(D, H, H^\vee)$ together with a skew-pairing $\ev\colon H^\vee\otimes H\rightarrow k$ such that $\br_H$, $\br_{H^\vee}$ and $\br_D$ defined in terms of $\ev$ satisfy the respective versions of \eqref{eq:Rmatrix}.
\end{remark}

\begin{remark}
We do not assume that $\ev$ is nondegenerate, so $H^\vee$ is not necessarily the linear dual to $H$.
\end{remark}

Define
\[\cC = \CoMod_D, \qquad \cD = \CoMod_H, \qquad \cE = \CoMod_{{H^\vee}^{op}}.\]

Using the skew-pairing $\ev$ we get a distributive law $\beta_{W, V}\colon W\otimes V\rightarrow V\otimes W$ for any $W\in\cE$ and $V\in\cD$ defined by
\[\beta_{W, V}(w\otimes v) = v_{(0)}\otimes w_{(0)} \ev(w_{(1)}, v_{(1)}).\]
Therefore, by \cref{prop:monoidaltensorproduct} we get a monoidal functor
\[\cE\otimes_{\cC} \cD\longrightarrow \Mod_k\]
and hence $(\CoMod_D, \CoMod_H, \CoMod_{(H^\vee)^{op}})$ becomes a quantum Manin triple.

\begin{remark}
The skew-pairing $\ev$ allows one to turn a right $H$-comodule $V$ into a left $H^\vee$-module via
\begin{equation}
h\triangleright v = v_{(0)} \ev(h, v_{(1)}).
\label{eq:comoduletomodule}
\end{equation}
This gives a monoidal functor
\begin{equation}
\CoMod_H\longrightarrow \LMod_{H^\vee}^{\mop}.
\label{eq:fouriertransform}
\end{equation}
Then the distributive law $\beta$ can be written as
\[\beta_{W, V}(w\otimes v) = (w_{(1)}\triangleright v)\otimes w_{(0)}.\]
\end{remark}

\subsection{Classical groups}
\label{sect:classicalManinTriple}

In this section we work out an example of a quantum Manin triple coming from an algebraic group.

Let $G$ be an affine group scheme over $k$ and $\g$ its Lie algebra which we assume is flat over $k$. Consider the symmetric monoidal category $\Rep G = \CoMod_{\cO(G)}$. $\U\g$ is a Hopf algebra in $\Rep G$, so $\cC=\CoMod_{\U\g}(\Rep G)$ is a monoidal category. We define the skew-pairing
\[\ev\colon \U\g\otimes \cO(G)\longrightarrow k\]
such that $\ev(x, f)$ is the derivative of $f\in\cO(G)$ along $x\in\mathfrak{g}$ at the unit. Then for $W\in\CoMod_{\U\g}$ and $V\in\Rep G$ we have a braiding isomorphism $W\otimes V\rightarrow V\otimes W$ given by
\begin{equation}
w\otimes v\mapsto (w_{(1)}\triangleright v)\otimes w_{(0)}.
\label{eq:Ugbraiding}
\end{equation}
This endows $\cC$ with a braiding and $\Rep G$ and $\CoMod_{\U\g}$ with structures of $\cC$-monoidal categories. The action functors $T\colon \cC\rightarrow \Rep G$ and $\tilde{T}\colon \cC\rightarrow \CoMod_{\U\g}$ are the forgetful functors.

\begin{prop} $ $
\begin{enumerate}
\item The functor $T^R\colon \Rep G\rightarrow \CoMod_{\U\g}(\Rep G)$ is given by the cofree $\U\g$-comodule $T^R(V) = \U\g\otimes V$, the counit $TT^R(V)\rightarrow V$ is given by the counit on $\U\g$ and the lax monoidal structure on $T^R(V)$ is given by the algebra structure on $\U\g$.

\item The functor $T^R\colon \Rep G\rightarrow \CoMod_{\U\g}(\Rep G)$ is monadic.

\item The functor $T\colon \CoMod_{\U\g}(\Rep G)\rightarrow \Rep G$ satisfies the projection formula.
\end{enumerate}
\end{prop}
\begin{proof}
By \cref{thm:hopfmodules} we may identify
\[\Rep G\cong \LMod_{\U\g}(\CoMod_{\U\g}(\Rep G))\]
via $V\mapsto \U\g\otimes V$. Under this identification the functor
\[T\colon \CoMod_{\U\g}(\Rep G)\rightarrow \Rep G\cong \LMod_{\U\g}(\CoMod_{\U\g}(\Rep G))\]
is given by $V\mapsto \U\g\otimes V$. Its right adjoint $T^R$ is the forgetful functor. This proves the first claim.

Clearly, the forgetful functor $T^R\colon \LMod_{\U\g}(\CoMod_{\U\g}(\Rep G))\rightarrow \CoMod_{\U\g}(\Rep G)$ is monadic with the corresponding monad on $\CoMod_{\U\g}(\Rep G)$ identified with $\U\g\otimes -$. This proves the second claim.

The morphism appearing in the projection formula is given by the composite
\[X\otimes (\U\g\otimes V)\longrightarrow \U\g\otimes(X\otimes (\U\g\otimes V))\longrightarrow \U\g\otimes X\otimes V\]
for $X\in \CoMod_{\U\g}(\Rep G)$ and $V\in \Rep G$. Here the first morphism is given by the $\U\g$-coaction and the second morphism is given by applying counit to the second copy of $\U\g$. Explicitly, for $x\in X$ and $v\in V$ we have
\[x\otimes h\otimes v\mapsto x_{(1)} h_{(1)} \otimes x_{(0)}\otimes h_{(2)}\otimes v\mapsto x_{(1)}h\otimes x_{(0)}\otimes v,\]
but this is precisely the isomorphism \eqref{eq:Ugbraiding} applied to the first two factors.
\end{proof}

\begin{example}
We see that $T^R(1) = \U\g\in \CoMod_{\U\g}(\Rep G)$. By \cref{prop:REAcommutative} it is a commutative algebra and let us show this explicitly. It is generated by $\g\subset \U\g$, so we need to check that the generators commute. For $x,y\in\g$ the braiding \eqref{eq:Ugbraiding} gives
\[x\otimes y\mapsto [x, y]\otimes 1 + y\otimes x.\]
Since $xy = [x, y] + yx$, this shows that $\U\g\in\cC$ is a commutative algebra.
\end{example}

From this proposition we obtain an isomorphism
\[\cF\cong \U\g.\]
For $V\in\Rep G$ the isomorphism
\[\tau_{T^R(1), V}\colon \cF\otimes V\xrightarrow{\sim} V\otimes \cF\]
given by \eqref{eq:Ugbraiding} has the following expression:
\[h\otimes v\mapsto (h_{(1)} \triangleright v)\otimes h_{(2)}.\]

Using \cref{prop:HCmonadic} we may then identify
\[\HC\cong \LMod_{\U\g}(\Rep G)\]
as monoidal categories. Given an object $V\in\LMod_{\U\g}(\Rep G)$ the isomorphism $\tau_{T^R(1), V}$ endows it with the right $\U\g$-module structure, so that we may identify $\LMod_{\U\g}(\Rep G)$ with the monoidal category of Harish--Chandra bimodules, i.e. $\U\g$-bimodules where the diagonal action of $\g$ is integrable.

We have the obvious forgetful functors $F\colon \Rep G\rightarrow \Mod_k$ and $\tilde{F}\colon \CoMod_{\U\g}\rightarrow \Mod_k$. The isomorphism \eqref{eq:Ugbraiding} provides a distributive law between them so that together they assemble into a monoidal functor
\[\CoMod_{\U\g}\otimes_\cC \Rep G\longrightarrow \Mod_k\]
which gives a quantum Manin triple $(\cC, \Rep G, \CoMod_{\U\g})$.

\begin{remark}
The quantum Manin triple $(\cC, \Rep G, \CoMod_{\U\g})$ is a quantization of the group triple $(T^* G, G, \g^*)$.
\end{remark}

\begin{remark}
Using \cref{thm:hopfmodules} we may identify $\Rep G\cong \LMod_{\U\g}(\cC)$. Therefore, by \cref{prop:modtensorproduct} we get an equivalence $\CoMod_{\U\g}\otimes_\cC \Rep G\cong \LMod_{\U\g}(\CoMod_{\U\g})$ and applying \cref{thm:hopfmodules} again we deduce that $\CoMod_{\U\g}\otimes_\cC \Rep G\cong \Mod_k$ which is easily seen to be the same functor as above.
\end{remark}

We have
\[F^R(k) = \cO(G)\in\Rep G,\qquad \tilde{F}^R(k) = \U\g\in \CoMod_{\U\g}\]
equipped with the obvious algebra structures.

The counit $\epsilon\colon \cF\rightarrow 1$ in $\Rep G$ is the counit of $\U\g$. For $W\in\CoMod_{\U\g}$ the counit $V\rightarrow \tilde{F}^R\tilde{F}(V)$ is given by the coaction map $V\rightarrow \U\g\otimes V$, so the map \eqref{eq:REAFRT} is given by the composite
\[\U\g\xrightarrow{\Delta} \U\g\otimes \U\g\xrightarrow{\id\otimes\epsilon} \U\g\]
which is the identity map.

The functor \eqref{eq:fouriertransform} in our case is the obvious symmetric monoidal functor
\[\Rep G\longrightarrow \LMod_{\U\g}.\]

\subsection{Quantum groups}
\label{sect:quantumgroups}

The standard references for quantum groups are \cite{Lusztig} and \cite{ChariPressley}. We will follow the categorical presentation from \cite[Sections 4, 5]{Gaitsgory} (see also \cite[Section 4.3]{Laugwitz}). Let $k=\C$.

We fix a connected reductive group $G$ with a choice of a Borel subgroup $B_+\subset G$ and a maximal torus $T\subset B_+\subset G$. Denote by $\cLambda$ the character lattice. In addition, we fix a symmetric bilinear Weyl-invariant form
\[b'\colon \cLambda\times\cLambda\rightarrow \C^\times.\]
Denote the associated quadratic form by
\[q(\lambda) = b'(\lambda, \lambda).\]
Our assumption will be that for each simple root $\calpha_i$ the value $q(\calpha_i)$ is not a root of unity.

Let $\Rep_q(T)$ be the usual monoidal category of $\cLambda$-graded vector spaces with braiding $k^{\lambda_1}\otimes k^{\lambda_2}\xrightarrow{\sim} k^{\lambda_2}\otimes k^{\lambda_1}$ of one-dimensional vector spaces concentrated in weights $\lambda_1,\lambda_2\in\cLambda$ given by multiplication by $b'(\lambda_1, \lambda_2)$. Denote by $\inv\colon \Rep_q(T)\rightarrow \Rep_q(T)$ the braided autoequivalence given by reversing the grading.

We have a bialgebra $\U_q(\n_+)\in\Rep_q(T)$ and denote $\U_q(\n_-) = \inv(\U_q(\n_+))$. In addition, we have a nondegenerate Hopf pairing
\[\ev\colon \U_q(\n_+)\otimes \U_q(\n_-)\longrightarrow k\]
in $\Rep_q(T)$.

Denote by $\cO_q(N_+)\in\Rep_q(T)^{\bop}$ the bialgebra $\U_q(\n_-)$ with the opposite multiplication. Then we may define
\[\Rep_q(B_+) = \CoMod_{\cO_q(N_+)}(\Rep_q(T)^{\bop})\]
which is naturally a $\Rep_q(T)^{\bop}$-monoidal category (see \cref{ex:comodulesbialgebra}). We may similarly define
\[\Rep_q(B_-) = \CoMod_{\U_q(\n_+)}(\Rep_q(T))\]
which is a $\Rep_q(T)$-monoidal category.

The relative Drinfeld center $\Z_{\Rep_q(T)^{\bop}}(\Rep_q(B_+))$ is given by $\cO_q(N_+)$-comodules $M$ in $\Rep_q(T)^{\bop}$ together with a natural isomorphism $\beta_N\colon M\otimes N\xrightarrow{\sim} N\otimes M$. In particular,
\[\cO_q(N_+)\otimes M\xrightarrow{\beta_{\cO_q(N_+)}^{-1}} M\otimes \cO_q(N_+)\xrightarrow{\id\otimes \epsilon} M\]
gives an $\cO_q(N_+)$-module structure on $M$. We denote by
\[\Rep_q(G)\subset \Z_{\Rep_q(T)^{\bop}}(\Rep_q(B_+))\]
the full subcategory where the $\cO_q(N_+)$-module structure is locally nilpotent, i.e. comes from an $\cO_q(N_-)$-coaction (see \cite[Lemma 4.3.5]{Gaitsgory}). We similarly have a fully faithful braided monoidal functor
\[\Rep_q(G)^{\bop}\subset \Z_{\Rep_q(T)}(\Rep_q(B_-)).\]

To summarize:
\begin{itemize}
\item $\Rep_q(G)$ is a braided monoidal category.

\item $\Rep_q(B_+)$ is a monoidal $\Rep_q(G)\otimes \Rep_q(T)^{\bop}$-category.

\item $\Rep_q(B_-)$ is a monoidal $\Rep_q(G)^{\bop}\otimes \Rep_q(T)$-category.
\end{itemize}

Consider the $\Rep_q(T)$-monoidal structure on $\Rep_q(B_-)$ given by precomposing the obvious one with $\inv$.

\begin{defn}
The category of representations of the \defterm{dual quantum group} is
\[\Rep_q(G^*) = \Rep_q(B_-)\otimes_{\Rep_q(T)} \Rep_q(B_+).\]
\label{def:dualquantumgroup}
\end{defn}

By construction $\Rep_q(G^*)$ is a $\Rep_q(G)\otimes \Rep_q(G)^{\bop}$-monoidal category. We will now construct a quantum Manin triple corresponding to the quantum group.

The forgetful functor $\Rep_q(T)\rightarrow \Mod_k$ is comonadic and this gives the Hopf algebra $\cO_q(T) = \cO(T)$. Similarly, the forgetful functors $\Rep_q(G)\rightarrow \Rep_q(T)\rightarrow \Mod_k$ and $\Rep_q(B_\pm)\rightarrow \Rep_q(T)\rightarrow \Mod_k$ are comonadic. This gives the following Hopf algebras:
\begin{itemize}
\item $\cO_q(B_\pm)$ with Hopf maps $p_{\pm}\colon \cO_q(B_\pm)\rightarrow \cO_q(T)$ and $i_{\pm}\colon \cO_q(T)\rightarrow \cO_q(B_\pm)$.

\item $\cO_q(G)$ with Hopf maps $f\colon \cO_q(G)\rightarrow \cO_q(B_+)$ and $g\colon \cO_q(G)\rightarrow \cO_q(B_-)$.
\end{itemize}

The braiding on $\Rep_q(T)$ corresponds to a coquasitriangular structure
\[\ev^T\colon \cO_q(T)\otimes \cO_q(T)\longrightarrow k.\]
In addition, we have a skew-pairing
\[\ev\colon \cO_q(B_-)\otimes \cO_q(B_+)\longrightarrow k\]
with the following properties (we denote by $\overline{\ev}$ the inverse skew-pairing):
\begin{enumerate}
\item $\br(x, y) = \ev(g(x), f(y))$ for $x,y\in\cO_q(G)$ gives a coquasitriangular structure on $\cO_q(G)$. Its inverse is $\overline{\br}(x, y) = \overline{\ev}(f(x), g(y))$.

\item $\br_{B_+}(x, y) = \ev(g(x), y)$ for $x\in\cO_q(G)$ and $y\in\cO_q(B_+)$ gives an $(\cO_q(G), \br)$-coquasitriangular structure on $\cO_q(B_+)$.

\item $\br_{B_-}(x, y) = \overline{\ev}(f(x), y)$ for $x\in\cO_q(G)$ and $y\in\cO_q(B_-)$ gives an $(\cO_q(G), \overline{\br})$-coquasitriangular structure on $\cO_q(B_-)$.

\item $\ev(i_-(x), y)=\ev^T(x, p_+(y))$ for every $x\in\cO_q(T)$ and $y\in\cO_q(B_+)$ and similarly for $B_-$.
\end{enumerate}

We refer to \cite[Proposition 6.34]{KlimykSchmudgen} for an explicit formula for the pairing $\ev$.

Note that the last property implies that
\begin{equation}
\br_{B_+}(x, i_+(y)) = \br_{B_-}(x, i_-(\inv(y)))
\label{eq:RmatrixT}
\end{equation}
for any $x\in\cO_q(G)$ and $y\in\cO_q(T)$.

Using the first three properties of $\ev$ we may construct a monoidal functor
\[(\Rep_q(B_-)\otimes \Rep_q(B_+))\otimes_{\Rep_q(G)\otimes \Rep_q(G)^{\bop}}\Rep_q G\cong \Rep_q(B_-)\otimes_{\Rep_q(G)} \Rep_q(B_+)\longrightarrow \Mod_k\]
as in \cref{sect:coalgebras}. The equation \eqref{eq:RmatrixT} shows that this monoidal functor descends to a monoidal functor
\[\Rep_q(G^*)\otimes_{\Rep_q(G)\otimes \Rep_q(G)^{\bop}} \Rep_q(G)\longrightarrow \Mod_k\]
which gives a factorizable quantum Manin triple $(\Rep_q(G)\otimes \Rep_q(G)^{\bop}, \Rep_q(G), \Rep_q(G^*))$.

\begin{remark}
The quantum Manin triple $(\Rep_q(G)\otimes \Rep_q(G)^{\bop}, \Rep_q(G), \Rep_q(G^*))$ is a quantization of the standard Manin triple $(G\times G, G, B_+\times_T B_-)$.
\end{remark}

We are now going to explain various algebraic structures associated to this quantum Manin triple in terms of some known constructions:
\begin{itemize}
\item The algebra $\cF\in\Rep_q(G)$ as an $\cO_q(G)$-comodule is equivalent to $\cO_q(G)$ equipped with the adjoint action of $\cO_q(G)$ on itself. Moreover, by \cref{prop:factorizableFRTREA} as a plain algebra it is isomorphic to $\cO_q(G)$. Note, however, that under the adjoint action $\cO_q(G)$ is not a comodule algebra; instead, as explained in \cite[Example 6.3]{BZBJ1} $\cF$ is the so-called reflection equation algebra.

\item The algebra $\U_q(\g) = \tilde{F}^R(k)\in\Rep_q(G^*)$ is the quantum group.

\item The homomorphism \eqref{eq:REAFRT} is the so-called Rosso homomorphism
\[\cO_q(G)\longrightarrow \U_q(\g),\]
see \cite[Proposition 10.16]{KlimykSchmudgen} for the case of matrix groups.

\item The functor
\[\Rep_q(G)\longrightarrow \LMod_{\U_q(\g)}\]
given by \eqref{eq:fouriertransform} realizes any object in $\Rep_q(G)$ as a module over the quantum group.
\end{itemize}

\begin{remark}
The constructions we have described in this section work for an arbitrary quantum parameter if we use modules over the Lusztig form of the quantum group for $\Rep_q(G)$. In this case $\tilde{F}^R(k)$ recovers the De Concini--Kac form $\U_q^{DK}(\g)$ of the quantum group.
\end{remark}

\section{Moment maps}
\label{sect:momentmaps}

In this section we define and study quantum moment maps for a quantum Manin pair.

\subsection{Quantum moment maps}

Let $(\cC, \cD)$ be a quantum Manin pair. For $V\in\cD$ we have a natural isomorphism
\[\tau_{T^R(1), V}\colon \cF\otimes V\xrightarrow{\sim} V\otimes \cF.\]

\begin{defn}
Let $A\in\cD$ be an algebra. A \defterm{quantum moment map} is an algebra map $\mu\colon \cF\rightarrow A$ such that the diagram
\[
\xymatrix{
\cF \otimes A \ar^{\tau_{T^R(1), A}}[dd] \ar^-{\mu\otimes\id}[r] & A\otimes A \ar^{m}[dr] & \\
&& A \\
A\otimes \cF \ar^-{\id\otimes\mu}[r] & A\otimes A \ar_{m}[ur] &
}
\]
commutes.
\label{def:quantummomentmap}
\end{defn}

\begin{remark}
By adjunction the algebra map $\mu\colon \cF\rightarrow A$ in $\cD$ is the same as an algebra map $T^R(1)\rightarrow T^R(A)$ in $\cC$. 
\end{remark}

\begin{example}
Recall from \cref{prop:REAcommutative} that $T^R(1)\in\cC$ is a commutative algebra. Therefore, the identity map $\cF\rightarrow \cF$ is a quantum moment map.
\end{example}

\begin{example}
Consider the quantum Manin pair $(\CoMod_{\U\g}(\Rep G), \Rep G)$ from \cref{sect:classicalManinTriple}. Since $\U\g$ is generated by $\g\subset \U\g$, it is enough to check the quantum moment map condition on $\g$. Using \eqref{eq:Ugbraiding} the isomorphism $\tau_{T^R(1), A}$ is $x\otimes a\mapsto x.a\otimes 1 + a\otimes x$. Then the quantum moment map condition is
\[\mu(x)a = x.a + a\mu(x).\]
In other words, $[\mu(x), a] = x.a$ which is the usual quantum moment map equation.
\end{example}

\begin{defn}
Let $\cM$ be a right $\cD$-module category and $M\in\cM$ an object. The \defterm{internal endomorphism object} $\uEnd_\cD(M)$ is the universal object in $\cD$ equipped with an isomorphism
\[\Hom_\cD(V, \uEnd_\cD(M))\cong \Hom_\cM(M\otimes V, M).\]
natural in $V\in\cD$.
\end{defn}

Note that if the internal endomorphism object $\uEnd_\cD(M)$ exists, it is naturally an algebra in $\cD$.

\begin{example}
Suppose $A\in\cD$ is an algebra, $\cM=\LMod_A(\cD)$ and consider the free $A$-module $M=A$. Then
\[\Hom_\cD(V, A)\cong \Hom_{\LMod_A(\cD)}(A\otimes V, A)\]
for every $V\in\cD$. In particular, $A\cong \uEnd_\cD(A)$.
\end{example}

Recall that there is a natural monoidal functor $\cD\rightarrow \LMod_\cF(\cD)$ sending $M\mapsto \cF\otimes M$. In particular, any $\LMod_\cF(\cD)$-module category is naturally a $\cD$-module category. The main observation about quantum moment maps is the following statement (see also \cite[Proposition 4.2]{BZBJ2}).

\begin{prop}
Suppose $A\in\cD$ is an algebra equipped with a quantum moment map $\mu\colon \cF\rightarrow A$. Then $\LMod_A(\cD)$ is a right $\LMod_\cF(\cD)$-module category. Conversely, suppose $\cM$ is a right $\LMod_\cF(\cD)$-module category and $M\in\cM$ an object admitting a $\cD$-internal endomorphism object $A=\uEnd_\cD(M)$. Then there is a quantum moment map $\mu\colon \cF\rightarrow A$.
\label{prop:momentmapHCmodule}
\end{prop}
\begin{proof}
Suppose $A$ is an algebra equipped with a quantum moment map $\mu\colon \cF\rightarrow A$. Let $M\in\LMod_A(\cD)$ be a left $A$-module. Using the quantum moment map $\mu$ we may turn $M$ into a left $\cF$-module. The quantum moment map equation then implies that the left $A$-module structure and the left $\cF$-module structures commute using $\tau$. Turning the left $\cF$-module structure into a right $\cF$-module structure, we see that $M$ canonically becomes an $(A, \cF)$-bimodule.

Therefore, we may define the action functor $\LMod_A(\cD)\otimes \LMod_\cF(\cD)\rightarrow \LMod_A(\cD)$ by
\[M\boxtimes V\mapsto M\otimes_{\cF} V.\] 

Conversely, suppose we have an action functor
\[\cM\otimes \LMod_\cF(\cD)\longrightarrow \cM\]
and consider an object $M\in\cM$ with $A=\uEnd(M)$. The unit of $\LMod_\cF(\cD)$ is $\cF$. Therefore, the image of $(M, \cF)$ under the functor is $M$. Considering endomorphism objects, this enhances the $(A, A)$-bimodule $A$ to an $(A, A\otimes \cF)$-bimodule, where the right $A$ and $\cF$-actions commute using $\tau$. The right $\cF$-module structure on $A$ induces a map $\mu\colon \cF\rightarrow A$ and the commutation of the right $A$ and $\cF$-module structures is exactly the quantum moment map equation.
\end{proof}

\begin{remark}
Note that the moment map condition is missing in the statement of \cite[Proposition 4.2]{BZBJ2}. In the notation of that reference, suppose $b$ is an algebra in $\cB$ equipped with an algebra map from $F^R(1_{\cB})$. The quantum moment map condition is precisely the condition that the induced left and right actions of $F^R(1_{\cB})$ are compatible, so $b$ descends from an algebra in $\cB$ to an algebra in $F^R(1_{\cB})$-modules in $\cB$. We are grateful to the authors for their correspondence confirming the correction.
\end{remark}

We will now define quantum Hamiltonian reduction in our context. Suppose $\mu\colon \cF\rightarrow A$ is a quantum moment map and let $\epsilon\colon \cF\rightarrow 1$ be the quantum moment map \eqref{eq:epsilon}. Then $A\otimes_\cF 1$ is naturally a left $A$-module. Consider the $k$-algebra
\[\Hom_{\LMod_A(\cD)}(A\otimes_\cF 1, A\otimes_\cF 1).\]

\begin{prop}
There is a natural isomorphism
\[\Hom_{\LMod_A(\cD)}(A\otimes_\cF 1, A\otimes_\cF 1)\cong \Hom_{\cD}(1, A\otimes_\cF 1).\]
\end{prop}
\begin{proof}
Note that the quantum moment map equation implies that $A$ is an algebra in
\[\LMod_\cF(\cD)\cong \RMod_\cF(\cD).\]

First, using the induction-restriction adjunction along the map $\epsilon\colon \cF\rightarrow 1$ we get an isomorphism
\[\Hom_{\LMod_A(\cD)}(A\otimes_\cF 1, A\otimes_\cF, 1)\cong \Hom_{\LMod_A(\RMod_\cF(\cD))}(A, A\otimes_\cF 1),\]
where we view $A\otimes_\cF 1$ as a right $\cF$-module via $\epsilon\colon \cF\rightarrow 1$. But the natural left $\cF$-module structure on $A\otimes_\cF 1$ is also via $\epsilon$. Therefore, we get an isomorphism
\[\Hom_{\LMod_A(\RMod_\cF(\cD))}(A, A\otimes_\cF 1)\cong \Hom_{\LMod_A(\cD)}(A, A\otimes_\cF 1).\]
Finally, using the induction-restriction adjunction along the unit map $1\rightarrow A$ we get an isomorphism
\[\Hom_{\LMod_A(\cD)}(A, A\otimes_\cF 1)\cong \Hom_{\cD}(1, A\otimes_\cF 1).\]
\end{proof}

\begin{defn}
Let $\mu\colon \cF\rightarrow A$ be a quantum moment map. The \defterm{Hamiltonian reduction of $A$} is the $k$-algebra
\[\Hom_{\cD}(1, A\otimes_\cF 1).\]
\label{def:Hamiltonianreduction}
\end{defn}

\begin{remark}
Using \cref{prop:momentmapHCmodule} we have the following interpretation of Hamiltonian reduction. Assume for simplicity that the conditions of \cref{prop:HCmonadic} are satisfied, so that we have a monoidal equivalence $\HC\cong\LMod_\cF(\cD)$. Recall from \cref{ex:trivialHCmodule} that $\cD$ is naturally a left $\HC$-module category. Then we have a pointed category
\[\LMod_A(\cD)\otimes_\HC \cD\]
where the pointing is given by $A\boxtimes 1$. It is easy to see that the endomorphism algebra of the pointing is exactly the Hamiltonian reduction of $A$, see e.g. \cite[Theorem 5.4]{BZBJ2}.
\end{remark}

\begin{example}
Let $X$ be a smooth affine algebraic variety with a free action of a connected reductive algebraic group $G$. Consider the algebra of differential operators $A=\D(X)\in\Rep(G)$. The infinitesimal action of $\g$ defines a map of Lie algebras $\g\rightarrow \Vect(X)\subset \D(X)$ which extends to a moment map $\mu\colon \U\g\rightarrow \D(X)$. The Hamiltonian reduction
\[\Hom_{\Rep(G)} (k, \D(X)\otimes_{\U\g} k)\cong (\D(X)/I)^G,\]
where $I\subset \D(X)$ is the left ideal generated by $\mu(v)$ for $v\in\g$, is isomorphic to $\D(X/G)$. We refer to \cite[Section 4]{EtingofCM} for more details on quantum Hamiltonian reduction.
\end{example}

\subsection{Monadic case}
\label{sect:quantummomentmapmonadic}

Suppose the conditions of \cref{prop:HCmonadic} are satisfied, so that we have a $\cC$-monoidal equivalence $\cD\cong \RMod_{T^R(1)}(\cC)$.

\begin{prop}
Suppose $\mu\colon \cF\rightarrow A$ is an algebra map in $\cD$ and let $\mu'\colon T^R(1)\rightarrow T^R(A)$ be the adjoint algebra map in $\cC$. Then $\mu$ is a quantum moment map iff $\mu'\colon T^R(1)\rightarrow T^R(A)$ is central.
\label{prop:momentmapcentral}
\end{prop}
\begin{proof}
By assumption $T^R$ is faithful, so the diagram in \cref{def:quantummomentmap} commutes in $\cD$ iff the pentagon in the diagram
\[
\xymatrix{
T^R(1)\otimes T^R(A) \ar[r] \ar^{\sigma}[dd] & T^R(TT^R(1)\otimes A) \ar^-{\mu\otimes\id}[r] \ar^{T^R(\tau_{T^R(1), A})}[dd] & T^R(A \otimes A) \ar^{T^R(m)}[dr] & \\
&&& T^R(A) \\
T^R(A)\otimes T^R(1) \ar[r] & T^R(A\otimes TT^R(1)) \ar^-{\id\otimes \mu}[r] & T^R(A\otimes A) \ar_{T^R(m)}[ur] &
}
\]
commutes in $\cC$. $T^R\colon \cD\rightarrow \cC$ is compatible with the $\cC$-monoidal structure, so the square on the left commutes as well. But the commutativity of the resulting diagram precisely expresses the condition that the map $T^R(1)\rightarrow T^R(A)$ is central.
\end{proof}

Let $A$ be any algebra in $\cD=\RMod_{T^R(1)}(\cC)$. Then we have an equivalence
\[\LMod_A(\cD)\cong \BiMod{A}{T^R(1)}(\cC).\]
In particular, it carries a natural right $\cD$-module structure.

Now suppose $A$ is equipped with a quantum moment map, i.e. a central map $\mu'\colon T^R(1)\rightarrow A$ in $\cC$. Let us explain how \cref{prop:momentmapHCmodule} works in this case, i.e. how to enhance the right $\cD$-module structure to an $\HC$-module structure. Recall from \cref{prop:HCmodule} that it means we need to provide an additional left $\cD$-module structure such that the resulting $\cC$-module structures are identified.

Given an $(A, T^R(1))$-bimodule $M$, using $\mu'\colon T^R(1)\rightarrow A$ we may turn it into a $(T^R(1), T^R(1))$-bimodule. Given another right $T^R(1)$-module $V$, the relative tensor product $V\otimes_{T^R(1)} M$ is still an $(A, T^R(1))$-bimodule since $T^R(1)\rightarrow A$ is central. This gives the required left $\cD$-action on $\LMod_A(\cD)$. The functor $T\colon \cC\rightarrow \cD$ is given by the free right $T^R(1)$-module construction, so the two $\cC$-actions are identified using the braiding.

\subsection{Coalgebras}

Consider a triple of Hopf algebras $(D, H, H^\vee)$ and the corresponding quantum Manin triple $(\CoMod_D, \CoMod_H, \CoMod_{(H^\vee)^{op}})$ from \cref{sect:coalgebras}. We denote by $\cF=T^R(1)$ the $D$-comodule algebra we have considered previously.

\begin{thm}
An algebra map $\mu\colon \cF\rightarrow A$ in $\CoMod_H$ is a quantum moment map iff
\[\mu(h) a = (h_{(1)}\triangleright a)\mu(h_{(0)})\]
for all $h\in H$ and $a\in A$, where $\Delta(h) = h_{(0)}\otimes h_{(1)}\in \cF\otimes H^\vee$ is the $H^\vee$-coaction on $\cF$.
\label{thm:quantummomentmapManintriple}
\end{thm}
\begin{proof}
Let $g\colon D\rightarrow H^\vee$ be the Hopf map in the definition of the triple $(D, H, H^\vee)$ and
\[\Delta(h) = h_{(0)}\otimes h_{(1)}'\in \cF\otimes D\]
be the $D$-coaction. Then $h_{(0)}\otimes h_{(1)} = h_{(0)}\otimes g(h_{(1)}')$.

The quantum moment map condition is
\[\mu(h) a = \br_A(h_{(1)}', a_{(1)}) a_{(0)}\mu(h_{(0)}).\]
The compatibility of $\br_A$ and the evaluation skew-pairing $\ev\colon H^\vee\otimes H\rightarrow k$ gives
\[\br_A(h_{(1)}', a_{(1)}) a_{(0)}\mu(h_{(0)}) = \ev(h_{(1)}, a_{(1)}) a_{(0)} \mu(h_{(0)})\]
and the definition of the $H^\vee$-action \eqref{eq:comoduletomodule} gives
\[\ev(h_{(1)}, a_{(1)}) a_{(0)} \mu(h_{(0)}) = (h_{(1)}\triangleright a) \mu(h_{(0)}).\]
\end{proof}

\begin{remark}
The previous proposition relates our notion of quantum moment maps to a more common one used in the literature, see e.g. \cite[Section 1.5]{VaragnoloVasserot}. In particular, the quantum Hamiltonian reduction defined in \cite[Theorem 1.5.2]{VaragnoloVasserot} coincides with \cref{def:Hamiltonianreduction}.
\end{remark}

\subsection{Fusion}

Fix a quantum Manin pair $(\cC, \cD)$ and suppose $\cM_1, \cM_2$ are two $\HC$-module categories. In particular, both are $(\cD, \cD)$-bimodule categories, so we may consider the relative tensor product $\cM_1\otimes_\cD \cM_2$. It is still an $(\cD, \cD)$-bimodule and it is clear that the resulting $\cC$-actions are identified. Thus, it becomes an $\HC$-bimodule.

\begin{defn}
Let $\cM_1, \cM_2$ be $\HC$-modules. Their \defterm{fusion} is the $\HC$-module
\[\cM_1\otimes_\cD \cM_2.\]
\end{defn}

\begin{remark}
Fusion gives a monoidal structure on the 2-category $\LMod_\HC$ of $\HC$-module categories. Consider the quantum Manin pair $(\cC\otimes \cC^{\bop}, \cC)$ for a braided monoidal category $\cC$. As explained in \cite[Section 3]{BZBJ2}, an $\HC$-module category is a braided module category, i.e. an $\bE_2$-module category over $\cC$. Then by \cite[Theorem 3.3.3.9]{HA} fusion may be upgraded to a braided monoidal structure on the 2-category $\LMod_\HC$.
\end{remark}

From now on we assume that the conditions of \cref{prop:HCmonadic} are satisfied, so that we may identify $\cD\cong \RMod_{T^R(1)}(\cC)$.

Suppose $A_1, A_2$ are two algebras in $\cD=\RMod_{T^R(1)}(\cC)$ equipped with central maps $\mu'_i\colon T^R(1)\rightarrow T^R(A_i)$. We have a natural algebra structure on $T^R(A_1)\otimes T^R(A_2)$ in $\cC$. Also, $T^R(A_1)$ and $T^R(A_2)$ are algebras in right $T^R(1)$-modules. Using $\mu'_2$ we induce a left $T^R(1)$-module structure on $T^R(A_2)$ so that by centrality it becomes an algebra with respect to the left $T^R(1)$-module structure as well. Therefore,
\[T^R(A_{fus}) = T^R(A_1)\otimes_{T^R(1)} T^R(A_2)\]
is an algebra in right $T^R(1)$-modules in $\cC$, i.e. an algebra in $\cD$. The map
\[\mu'_{fus}\colon T^R(1)\rightarrow T^R(A_1)\otimes_{T^R(1)} T^R(A_2)\]
given by $t\mapsto \mu_1'(t)\otimes 1$ is still central, so it defines a moment map $\mu_{fus}\colon \cF\rightarrow A_{fus}$.

\begin{prop}
Suppose $A_1, A_2$ are two algebras in $\cD$ equipped with quantum moment maps $\mu_i\colon \cF\rightarrow A$. Then the fusion of $\LMod_{A_1}(\cD)$ and $\LMod_{A_2}(\cD)$ is equivalent to the category $\LMod_{A_{fus}}(\cD)$. Moreover, this equivalence respects the natural pointings on both sides.
\end{prop}
\begin{proof}
By \cref{prop:modtensorproduct} the functor
\[\LMod_{A_1}(\cD)\otimes_{\cD} \LMod_{A_2}(\cD)\longrightarrow \LMod_{A_1}(\LMod_{A_2}(\cD))\]
is an equivalence, where the left $\cD$-module structure on $\LMod_{A_2}(\cD)$ is given in terms of $\mu_2'$ as described in \cref{sect:quantummomentmapmonadic}. Both categories $\LMod_{A_1}(\LMod_{A_2}(\cD))$ and $\LMod_{A_1\otimes_{T^R(1)} A_2}(\cD)$ are monadic over $\cD$ and the corresponding monads are given by $A_1\otimes_{T^R(1)} A_2\otimes_{T^R(1)}(-)$
\end{proof}

\begin{example}
Consider the quantum Manin pair $(\CoMod_{\U\g}(\Rep G), \Rep G)$ from \cref{sect:classicalManinTriple}. Recall that for $A\in\Rep G$ the functor $T^R$ is given by $T^R(A) = A\otimes \U\g$, the cofree right $\U\g$-comodule. $T^R(A)$ carries a natural $(T^R(1)=\U\g)$-module structure given by the right $\U\g$-action on itself. If $\mu\colon \U\g\rightarrow A$ is a quantum moment map in $\Rep G$, then its adjoint is
\[\mu'\colon \U\g\rightarrow A\otimes \U\g\]
given by $h\mapsto \mu(h_{(1)})\otimes h_{(2)}$.

Now suppose $A_1, A_2\in\Rep G$ are two algebras and $\mu_i\colon \U\g\rightarrow A_i$ are two quantum moment maps. Then
\[T^R(A_{fus}) = T^R(A_1)\otimes_{T^R(1)} T^R(A_2)\]
as an object of $\cC$ equipped with the right $T^R(1)$-module structure from the right $T^R(1)$-action on $T^R(A_2)$. In our case we get
\[T^R(A_{fus}) = (A_1\otimes \U\g)\otimes_{\U\g}(A_2\otimes \U\g)\cong A_1\otimes A_2\otimes \U\g,\]
i.e. $A_{fus}\cong A_1\otimes A_2\in\Rep G$. The algebra structure on $T^R(A_{fus})$ is uniquely determined by the condition that $T^R(A_1)\otimes T^R(A_2)\rightarrow T^R(A_{fus})$ is an algebra map and a computation shows that the algebra structure on $A_{fus}\cong A_1\otimes A_2$ is the pointwise product.

For $h\in\U\g$ we have
\[\mu'_{fus}(h) = (\mu_1(h_{(1)}) \otimes h_{(2)})\otimes (1\otimes 1)\in (A_1\otimes \U\g)\otimes_{\U\g}(A_2\otimes \U\g).\]
Under the isomorphism with $A_1\otimes A_2\otimes \U\g$ it corresponds to
\[\mu'_{fus}(h) = \mu_1(h_{(1)})\otimes \mu_2(h_{(2)})\otimes h_{(3)}\in A_1\otimes A_2\otimes \U\g\]
and hence the moment map $\mu_{fus}\colon \U\g\rightarrow A_1\otimes A_2$ is given by $h\mapsto \mu_1(h_{(1)})\otimes \mu_2(h_{(2)})$.
\end{example}

\begin{example}
Suppose $\cD$ is a braided monoidal category and consider the quantum Manin pair $(\cD\otimes \cD^{\bop}, \cD)$. Consider an algebra $A\in\cD$ equipped with a central map $\mu'\colon T^R(1)\rightarrow T^R(A)$. Denote by $u\colon T^R(1)\rightarrow T^R(A)$ the image of the unit map under $T^R(A)$. An $\HC$-module category is a $(\cD, \cD)$-bimodule category with two identifications of the $\cD$-module structure. Then we get an equivalence of the $(T^R(1), T^R(1))$-bimodules
\[{}_{\mu'}T^R(A)_u\cong {}_u T^R(A)_u.\]
So, we may identify the fusion as
\[T^R(A_{fus}) = {}_{\mu'_1} T^R(A_1)_u \otimes_{T^R(1)} {}_{\mu'_2} T^R(A_2)_u\cong {}_{\mu'_1} T^R(A_1)_u \otimes_{T^R(1)} {}_u T^R(A_2)_u\cong T^R(A_1\otimes A_2),\]
i.e. $A_{fus}$ is equivalent to the usual tensor product $A_1\otimes A_2$ of algebras in $\cD$.
\end{example}

\section{Classical moment maps}
\label{sect:classicalsetting}

In this section we define moment maps in the classical setting and explain how they appear as classical degenerations of quantum moment maps.

\subsection{Quasi-Poisson geometry}

Let $G$ be an algebraic group.

\begin{defn}
A \defterm{quasi-Poisson structure} on $G$ is a bivector $\pi\in\wedge^2 \T_G$ and a trivector $\phi\in\wedge^3(\g)$ such that
\begin{enumerate}
\item $\pi$ is multiplicative.

\item $\frac{1}{2}[\pi, \pi] = \phi^L - \phi^R$.

\item $[\pi, \phi^R] = 0$.
\end{enumerate}
\end{defn}

Since $\pi$ is multiplicative, it vanishes at the unit $e\in G$. In particular, its linear part at the unit defines a Lie cobracket $\delta\in \g^*\otimes \wedge^2(\g)$. We denote by $\llbracket-, -\rrbracket$ the natural Lie bracket on $\wedge^\bullet(\g^*)\otimes\wedge^\bullet(\g)$ given by pairing the $\g^*$ and $\g$ factors.

\begin{defn}
Let $t\in\wedge^2(\g)$ and suppose $G$ is a quasi-Poisson group. The \defterm{twist} of $G$ is a new quasi-Poisson structure on $G$ given by
\begin{align*}
\pi' &= \pi + t^L - t^R \\
\phi' &= \phi + \frac{1}{2}\llbracket\delta, t\rrbracket + \frac{1}{2}[t, t].
\end{align*}
\end{defn}

\begin{defn}
Let $G$ be a quasi-Poisson group. A \defterm{quasi-Poisson $G$-variety} is a $G$-variety $X$ equipped with a bivector $\pi_X$ such that
\begin{enumerate}
\item The action map $a\colon G\times X\rightarrow X$ is compatible with the bivectors.

\item $\frac{1}{2}[\pi_X, \pi_X] = a(\phi)$.
\end{enumerate}
\end{defn}

\begin{remark}
If $\phi = 0$, we say $G$ is a Poisson-Lie group. A quasi-Poisson $G$-variety in this case is a Poisson $G$-variety.
\end{remark}

\begin{lm}
Suppose $G$ is a quasi-Poisson group and $(X, \pi_X)$ is a quasi-Poisson $G$-variety and fix $t\in\wedge^2(\g)$. Let $G'$ be the twist of $G$ with respect to $t$. Then $(X, \pi_X - a(t))$ is a quasi-Poisson $G'$-variety.
\end{lm}

Quasi-Poisson groups usually appear in the following way.

\begin{defn}
A \defterm{group pair} is a pair of algebraic groups $(D, G)$ where $G\subset D$ is a closed subgroup together with a nondegenerate element $c\in\Sym^2(\fd)^D$ such that $\g\subset \fd$ is Lagrangian.
\end{defn}

\begin{defn}
A \defterm{quasi-triple} is a group pair $(D, G)$ together with a Lagrangian complement $\h\subset \fd$ to $\g$ which we do not assume is a Lie subalgebra.
\end{defn}

Denote by $p_\h\colon \fd^*\cong \fd\rightarrow \g$ the projection determined by the complement.

\begin{example}
If $G$ is equipped with a nondegenerate element $c\in\Sym^2(\g)^G$, we have a group pair $(G\times G, G)$ where $G\subset G\times G$ is equipped diagonally and the pairing on $\fd=\g\oplus\g$ is the difference of $c$ on each summand.
\label{ex:factorizablegrouppair}
\end{example}

\begin{remark}
The complementarity condition implies that the composite $\h\rightarrow \fd\rightarrow \fd/\g$ is an isomorphism and the Lagrangian condition for $\g$ implies that $\g^*\rightarrow \fd/\g$ induced by $c$ is an isomorphism.
\end{remark}

Fix a group pair $(D, G)$. Any two Lagrangian complements $\h_1, \h_2\subset \fd$ to $\g\subset \fd$ differ by a twist $t\in\wedge^2(\g)$ as described in \cite[Section 2.2]{AKS}. The difference $p_{\h_2}-p_{\h_1}$ of the two projections $\fd^*\rightarrow \g$ is then given by the composite
\begin{equation}
\fd^*\longrightarrow \g^*\xrightarrow{-t} \g.
\label{eq:complementprojectiondifference}
\end{equation}

\begin{defn}
A \defterm{group triple} is a triple of algebraic groups $(D, G, G^*)$ such that $(D, G)$ and $(D, G^*)$ are group pairs and $\g^*\subset \fd$ and $\g\subset \fd$ are complementary.
\end{defn}

\begin{example}
For any group $G$ the triple $(\T^* G, G, \g^*)$ is a group triple, where $\Lie(\T^* G) = \g\oplus \g^*$ has the obvious pairing between $\g$ and $\g^*$ and $\g^*\subset \T^* G$ is an abelian subgroup.
\end{example}

\begin{example}
If $G$ is a semisimple group with a choice of a Borel subgroup $B_+\subset G$ and a maximal torus $T\subset B_+$. Consider $D=G\times G$ and equip its Lie algebra $\fd=\g\oplus \g$ with the difference of the Killing forms on each summand. Then the diagonal embedding $G\subset D$ provides a group pair. Let $B_-\subset G$ be the opposite Borel subgroup and denote by $p_{\pm}\colon B_{\pm}\rightarrow T$ the projections. Define the subgroup $G^* =\{b_+\in B_+,\ b_-\in B_-\ |\ p_+(b_+)p_-(b_-)=1\}\subset B_+\times B_-$. Then $(G\times G, G, G^*)$ is a group triple.
\end{example}

Given a quasi-triple $(D, G, \h)$ it is shown in \cite[Section 3]{AKS} that we obtain natural quasi-Poisson structures $\pi_D^\h$ and $\pi_G^\h$ on $D$ and $G$ so that $G\subset D$ is a quasi-Poisson map. If $\h\subset \fd$ is a Lie subalgebra, then these are in fact Poisson-Lie structures. Moreover, changing the Lagrangian complement $\h$ to another one differing by $t\in\wedge^2(\g)$ corresponds to twisting the quasi-Poisson structures $\pi_D^\h$ and $\pi_G^\h$ by $t$.

So, for a group pair $(D, G)$ we will say a $G$-variety $X$ is a quasi-Poisson $G$-variety if it is equipped with a family of bivectors $\pi_X^\h$ for any complement $\h\subset \fd$ such that:
\begin{itemize}
\item $(X, \pi_X^\h)$ is a quasi-Poisson $(G, \pi_G^\h)$-variety.

\item For any two complements $\h_1, \h_2\subset \fd$ differing by a twist $t\in\wedge^2(\g)$ we have
\[\pi_X^{\h_2} = \pi_X^{\h_1} - a(t).\]
\end{itemize}

\subsection{Classical moment maps}
\label{sect:classicalmomentmaps}

Consider the $D$-space $D/G$ via the left $D$-action on itself. We may also restrict it to a $G$-action on $D/G$. Let $\{e_i\}$ be a basis of $\fd$ and $\{e^i\}$ the dual basis of $\fd^*$. We denote by $\tilde{a}\colon \fd\rightarrow \Gamma(D/G, \T_{D/G})$ the infinitesimal $D$-action.

\begin{defn}
Let $(D, G)$ be a group pair and $X$ a quasi-Poisson $G$-space. A $G$-equivariant map $\mu\colon X\rightarrow D/G$ is a \defterm{moment map} if for some Lagrangian complement $\h\subset \fd$
\[\{\mu^* f_1, f_2\}_\h = \sum_i \mu^*(\tilde{a}(e_i). f_1) \cdot a(p_\h(e^i)). f_2\]
for every $f_1\in\cO_G$ and $f_2\in\cO_X$.
\label{def:classicalmomentmap}
\end{defn}

The following statement is a version of \cite[Proposition 5.1.5]{AKS}.

\begin{prop}
If the moment map condition in \cref{def:classicalmomentmap} is satisfied for some Lagrangian complement $\h\subset \fd$, it is satisfied for any Lagrangian complement.
\end{prop}
\begin{proof}
The moment map condition is equivalent to the commutativity of the diagram
\[
\xymatrix{
\T^*_{X, x} \ar^{\pi_X^\h}[rr] && \T_{X, x} \\
\T^*_{D/G, \mu(x)} \ar^-{\tilde{a}^*}[r] \ar^{\mu^*}[u] & \fd^* \ar^{p_\h}[r] & \g \ar^{a}[u]
}
\]
for every $x\in X$.

Suppose $\h'\subset \fd$ is another Lagrangian complement differing by a twist $t\in\wedge^2(\g)$ and write
\[t = \frac{1}{2} \sum_{i, j} t^{ij} x_i\wedge x_j\]
where $\{x_i\}$ is a basis of $\g$. Then $\{x_i, x^i\}$ is a basis of $\fd$ where $x^i$ is the dual basis of $\g^*\cong \h$. Since $\pi_X^{\h'} = \pi_X^\h - a(t)$, we get
\begin{align*}
\{\mu^* f_1, f_2\}_{\h'} &= \{\mu^* f_1, f_2\}_\h - \sum_{i, j} t^{ij} a(x_i).\mu^* f_1\cdot a(x_j).f_2 \\
&= \sum_j \mu^*(\tilde{a}(x^j).f_1)\cdot a(x_j).f_2 - \sum_{i, j} t^{ij} a(x_i).\mu^* f_1\cdot a(x_j).f_2 \\
&= \sum_j \mu^*(\tilde{a}(x^j).f_1)\cdot a(x_j).f_2 - \sum_{i, j} t^{ij} \mu^*(\tilde{a}(x_i). f_1)\cdot a(x_j).f_2 \\
&= \sum_{i, j} \mu^*(\tilde{a}(x^j - t^{ij}x_i).f_1)\cdot a(x_j).f_2
\end{align*}
where we use $G$-equivariance of the moment map in the third line. But by \eqref{eq:complementprojectiondifference} this is exactly the right-hand side of the moment map equation for $\h'$.
\end{proof}

We will now relate \cref{def:classicalmomentmap} to other moment maps appearing in the literature.

\begin{defn}
Let $(D, G)$ be a group pair. A Lagrangian complement $\h\subset \fd$ is \defterm{admissible} at $s\in D/G$ if the map $\h\rightarrow \fd\xrightarrow{\tilde{a}} \T_{D/G, s}$ is an isomorphism.
\end{defn}

If $U\subset D/G$ is an open subset where a Lagrangian complement $\h\subset \fd$ is admissible, we have a one-form $\theta^\h\in\Omega^1(U; \g^*)$ defined using the isomorphism $\T^*_{D/G, s}\cong \h^*\cong \g$. The following definition of moment maps is given in \cite[Definition 5.1.1]{AKS}.

\begin{defn}
Let $(D, G)$ be a group pair and $X$ a quasi-Poisson $G$-variety. A $G$-equivariant map $\mu\colon X\rightarrow D/G$ is a \defterm{moment map} if for every open $U\subset D/G$ and an admissible complement $\h\subset \fd$ on $U$ the equation
\[a(x) = (\pi_X^\h)^\sharp(\mu^*\langle \theta^\h, x\rangle)\]
holds on $U$ for every $x\in\g$.
\label{def:AKSmomentmap}
\end{defn}

\begin{remark}
For the group pair $(G\times G, G)$ from \cref{ex:factorizablegrouppair} we have $D/G\cong G$ as $G$-spaces where the $G$-action on the right is given by conjugation. In this case the notion of a moment map $\mu\colon X\rightarrow G$ is closely related to the notion of group-valued quasi-symplectic moment maps introduced in \cite{AMM} as explained in \cite{AKSM} and \cite{LBS}.
\end{remark}

\begin{prop}
Let $(D, G)$ be a group pair and $X$ a quasi-Poisson $G$-variety. A $G$-equivariant map $\mu\colon X\rightarrow D/G$ is a moment map in the sense of \cref{def:classicalmomentmap} iff if it is a moment map in the sense of \cref{def:AKSmomentmap}.
\label{prop:AKSequivalence}
\end{prop}
\begin{proof}
We may check the moment map condition in \cref{def:classicalmomentmap} on an open cover. Consider an open set $U\subset D/G$ from this cover an admissible Lagrangian complement $\h\subset \fd$ on $U$. For $x\in\mu^{-1}(U)$ the moment map condition \cref{def:AKSmomentmap} is equivalent to the commutativity of the diagram
\[
\xymatrix{
\T^*_{X, x} \ar^{(\pi_X^\h)^\sharp}[r] & \T_{X, x} \\
\T^*_{D/G, \mu(x)} \ar^{\mu^*}[u] & \g \ar_-{\theta^\h}[l] \ar^{a}[u]
}
\]
Since $\theta$ is an isomorphism, this diagram commutes iff the diagram
\[
\xymatrix{
\T^*_{X, x} \ar^{(\pi_X^\h)^\sharp}[rr] && \T_{X, x} \\
\T^*_{D/G, \mu(x)} \ar^{\mu^*}[u] \ar^-{\tilde{a}^*}[r] & \fd^* \ar^{p_\h}[r] & \g \ar^{a}[u]
}
\]
commutes which is precisely the moment map condition in \cref{def:classicalmomentmap}.
\end{proof}

Let us now fix a group triple $(D, G, G^*)$. Then we get a $G^*$-action on $D/G$ coming from the inclusion $G^*\subset D$. The following is well-known.

\begin{lm}
The composite
\[f\colon G^*\longrightarrow D\longrightarrow D/G\]
is \'{e}tale.
\label{lm:Rossoetale}
\end{lm}
\begin{proof}
Since the map is locally of finite presentation, it is enough to show that it is formally \'{e}tale. Since $G^*\rightarrow D/G$ is $G^*$-equivariant, it is enough to check that it is formally \'{e}tale at the unit $e\in G^*$. The map on tangent spaces at $e\in G^*$ is
\[\g^*\longrightarrow \fd\longrightarrow \fd/\g\]
which is an isomorphism since $\g^*\subset \fd$ is a complementary Lagrangian to $\g\subset \fd$.
\end{proof}

\begin{remark}
Consider the quantum Manin triple $(\Rep(D), \Rep(G), \Rep(G^*))$ where every category is symmetric monoidal and the action functors $T\colon \Rep(D)\rightarrow \Rep(G)$ and $\tilde{T}\colon \Rep(D)\rightarrow \Rep(G)$ are the obvious forgetful functors. Then the morphism \eqref{eq:REAFRT} is $\cO(D/G)\rightarrow \cO(G^*)$.
\end{remark}

In the case of group triples, the moment map equation \cref{def:classicalmomentmap} may be written as
\begin{equation}
\{\mu^* f_1, f_2\} = \sum_i \mu^*(\tilde{a}(x^i).f_1) a(x_i).f_2,
\end{equation}
where $\{x_i\}$ is a basis of $\g$ and $\{x^i\}$ is the dual basis of $\g^*$.

The coadjoint action of $\g$ on $\g^*$ gives a $\g$-action on $G^*$ such that $G^*\rightarrow D/G$ is $\g$-equivariant. Let $\theta\in\Omega^1(G^*; \g^*)$ be the left-invariant Maurer--Cartan form. The Lagrangian complement $\g^*\subset \fd$ is admissible for every $f(g)$ and we have
\begin{equation}
f^*\theta^{\g^*} = \theta.
\label{eq:LuAKS}
\end{equation}

The following notion of moment map was introduced in \cite{LuClassical}.

\begin{defn}
Let $X$ be a Poisson $G$-variety. A $\g$-equivariant map $X\rightarrow G^*$ is a \defterm{moment map} if
\[a(x) = \pi_X^\sharp(\mu^*\langle \theta, x\rangle)\]
for every $x\in\g$.
\label{def:Lumomentmap}
\end{defn}

\begin{lm}
Let $(D, G, G^*)$ be a group triple, $X$ a Poisson $G$-variety and $\mu\colon X\rightarrow D/G$ a $G$-equivariant map.
\begin{enumerate}
\item If $\mu$ satisfies the moment map condition of \cref{def:AKSmomentmap}, then $\mu^{-1}(G^*)\rightarrow G^*$ satisfies the moment map condition of \cref{def:Lumomentmap}.

\item If $D/G$ is irreducible and $\mu^{-1}(G^*)\rightarrow G^*$ satisfies the moment map condition of \ref{def:Lumomentmap}, then $\mu\colon X\rightarrow D/G$ satisfies the moment map condition \cref{def:AKSmomentmap}.
\end{enumerate}
\label{lm:LuAKS}
\end{lm}
\begin{proof}$ $
\begin{enumerate}
\item On the image of $G^*$ the complement $\g^*\subset \fd$ is admissible and so the result follows from \eqref{eq:LuAKS}.

\item By \cref{lm:Rossoetale} the map $G^*\rightarrow D/G$ is \'{e}tale and hence open. Since $D/G$ is irreducible, $f(G^*)\subset D/G$ is dense and so the moment map condition may be checked by pulling back to $G^*$.
\end{enumerate}
\end{proof}

\begin{example}
Consider the group triple $(T^* G, G, \g^*)$. The Poisson-Lie structure on $G$ is zero, so a Poisson $G$-variety is a Poisson variety $X$ equipped with a $G$-action which preserves the Poisson structure. The map $\g^*\rightarrow D/G$ is an isomorphism, so the moment map conditions \cref{def:AKSmomentmap} and \cref{def:Lumomentmap} are equivalent. In this case the moment map condition for a $G$-equivariant map $\mu\colon X\rightarrow \g^*$ reduces to
\[a(x) = \pi_X^\sharp(\d\mu(x)),\]
which is the usual moment map condition.
\end{example}

\subsection{Classical limit}
\label{sect:classicallimit}

In this section we show that classical degenerations of quantum moment maps given by \cref{def:quantummomentmap} produce classical moment maps in the sense of \cref{def:classicalmomentmap}. Let $\cA=k[\hbar]/\hbar^2$.

Let $D$ be an affine algebraic group and suppose $\cO_\hbar(D)$ is a coquasitriangular Hopf algebra flat over $\cA$ together with an isomorphism of coquasitriangular Hopf algebras $\cO_\hbar(D)/\hbar\cong \cO(D)$. Let
\[r = (\br - \epsilon\otimes\epsilon)/\hbar\colon \cO(D)\otimes \cO(D)\rightarrow k\]
and
\[\{a, b\}_D=\frac{ab-ba}{\hbar}\]
the biderivation on $\cO(D)$. Standard arguments (see \cite[Proposition 9.5]{ES} for the dual version) show the following:
\begin{itemize}
\item $r\in\fd\otimes \fd$.

\item The symmetric part $c$ of $r$ is $\fd$-invariant.

\item The biderivation $\{-, -\}_D$ is multiplicative.

\item $\{-, -\}_D$ is the biderivation on $D$ induced by $r$, i.e.
\[\{a,b\}_D = r(a_{(2)}, b_{(2)})b_{(1)}a_{(1)} - r(a_{(1)}, b_{(1)})a_{(2)}b_{(2)}\]
for every $a,b\in\cO(D)$.
\end{itemize}

In addition, consider a closed subgroup $G\subset D$ and a $\cO_\hbar(D)$-coquasitriangular Hopf algebra $\cO_\hbar(G)$ flat over $\cA$ together with an isomorphism of $\cO(D)$-coquasitriangular Hopf algebras $\cO_\hbar(G)/\hbar\cong \cO(G)$. Denote the $\cO_\hbar(D)$-coquasitriangular Hopf structure on $\cO_\hbar(G)$ by $\br_G$ and let
\[r_G = (\br_G - \epsilon\otimes\epsilon)/\hbar\colon \cO(D)\otimes \cO(G)\rightarrow k\]
and
\[\{a, b\}_G=\frac{ab-ba}{\hbar}\]
the biderivation on $\cO(G)$. Then analogously we get:
\begin{itemize}
\item $r_G=r\in\fd\otimes \g$. In particular, $\g\subset \fd$ is coisotropic with respect to $c$.

\item The biderivation $\{-, -\}_G$ is multiplicative.

\item $\{-, -\}_G$ is the biderivation on $G$ induced by $r$, i.e.
\[\{f(d), h\}_G = r(d_{(2)}, h_{(2)})h_{(1)}f(d_{(1)}) - r(d_{(1)}, h_{(1)})f(d_{(2)}) h_{(2)}\]
for every $d\in\cO(D)$ and $h\in\cO(G)$. Since $\cO(D)\rightarrow \cO(G)$ is surjective, this uniquely determines $\{-, -\}_G$.
\end{itemize}

Next, consider a $G$-variety $X$ and suppose $A_\hbar$ is an $\cO_\hbar(G)$-comodule algebra flat over $\cA$ together with an isomorphism of $G$-representations $A_\hbar/\hbar\cong \cO(X)$. Then 
\[\{a, b\}_X=\frac{ab-ba}{\hbar}\]
is a biderivation on $\cO(X)$ such that the coaction $\cO(X)\rightarrow \cO(X)\otimes \cO(G)$ is compatible with the brackets on both sides.

\begin{example}
Let $\Rep_\hbar(D) = \CoMod_{\cO_{\hbar}(D)}(\Mod_\cA)$ and $\Rep_\hbar(G) = \CoMod_{\cO_{\hbar}(G)}(\Mod_\cA)$. The forgetful functor $T\colon \Rep_\hbar(D)\rightarrow \Rep_\hbar(G)$ is monoidal, so its right adjoint $T^R$ is lax monoidal. Therefore, $T^R(1)\in\Rep_\hbar(D)$ is an algebra object. At $\hbar = 0$ we get $\cO(D/G)$ as an $\cO(D)$-comodule algebra. Recall also the algebra $\cF = TT^R(1)$ which at $\hbar=0$ is $\cO(D/G)$ viewed as an $\cO(G)$-comodule algebra.

The existence of the algebra map $\epsilon\colon \cF = TT^R(1)\rightarrow 1$ implies that the bracket $\{-, -\}_{D/G}$ vanishes at the basepoint $e\in D/G$. Therefore, the coaction $\cO(D/G)\rightarrow \cO(D)$ on $e\in D/G$ is compatible with the brackets and is injective, so $\{-, -\}_{D/G}$ is uniquely determined by $\{-, -\}_D$. See also \cite[Section 3.5]{AKS} for an explicit description of this bracket.
\end{example}

Suppose $c\in\Sym^2(\fd)^D$ is nondegenerate and $G\subset D$ is Lagrangian (rather than just coisotropic). Then $(D, G)$ is a group pair and we may talk about moment maps. Consider the map $r_2\colon \h=\g^*\rightarrow \fd$ given by pairing with the second component of $r$.

\begin{lm}
The map $r_2\colon \h\rightarrow \fd$ is injective and is a Lagrangian embedding complementary to $\g\subset \fd$. Moreover,
\[r \in\g^*\otimes \g\subset \fd\otimes \g\]
is the canonical element.
\end{lm}
\begin{proof}
The Lagrangian condition implies that the map
\[r_1\colon (\fd/\g)^*\longrightarrow \g\]
given by pairing with the first component of $r$ is an isomorphism. Its dual is $\g^*\xrightarrow{r_2} \fd\rightarrow \fd/\g$ which is therefore also an isomorphism. So, $\h\rightarrow \fd$ is injective and complementary to $\g\subset \fd$. Therefore, $r_2\colon \g^*\rightarrow \h$ is an isomorphism onto its image, so
\[r = \h\otimes \g\subset \fd\otimes \g\]
is the canonical element. In particular, $\h\subset \fd$ is also Lagrangian.
\end{proof}

So, in our setting we obtain a quasi-triple $(D, G, \h)$. Moreover, $r_1\colon \fd^*\rightarrow \g$ is exactly the map $p_\h\colon \fd^*\rightarrow \g$ induced by $\h\subset \fd$.

\begin{remark}
Conversely, given a quasi-triple $(D, G, \h)$ Alekseev and Kosmann-Schwarzbach \cite[Section 2.3]{AKS} define the canonical $r$-matrix on $\fd$ to be the canonical element $\g^*\otimes \g\subset \fd\otimes \g$.
\end{remark}

\begin{thm}
Suppose $\mu\colon X\rightarrow D/G$ is a $G$-equivariant map and $\mu_\hbar\colon \cF\rightarrow A_\hbar$ is a quantum moment map which is $\mu^*$ modulo $\hbar$. Then $\mu$ is a classical moment map.
\label{thm:momentmapclassicaldegeneration}
\end{thm}
\begin{proof}
Choose identifications
\[\cO_\hbar(D)\cong \cO(D)\otimes \cA,\qquad \cO_\hbar(G)\cong \cO(G)\otimes \cA,\qquad A_\hbar\cong \cO(X)\otimes \cA.\]
With respect to these identifications decompose the star product on $A_\hbar$ as
\[a\ast b = ab + \hbar B_1(a, b),\]
the moment map as
\[\mu_\hbar(a) = \mu^*(a) + \hbar \mu_1(a)\]
and the coquasitriangular structure as
\[\br = \epsilon\otimes \epsilon + \hbar r.\]
Let $\{e_i\}$ be a basis of $\fd$. Then the quantum moment map equation for $f_1\in \cO(D/G)$ and $f_2\in \cO(X)$ is
\begin{align*}
\mu^*(f_1) f_2 + \hbar B_1(\mu^*(f_1), f_2) + \hbar \mu_1(f_1) f_2 &= f_2\mu^*(f_1) + \hbar B_1(f_2, \mu^*(f_1)) \\
&+ \hbar f_2\mu_1(f_1) + \sum_i \hbar a(p_\h(e^i)).f_2 \mu^*(\tilde{a}(e_i).f_1).
\end{align*}
The only nontrivial equation is at the order $\hbar$ and we get
\[\{\mu^*(f_1), f_2\} = \sum_i \mu^*(\tilde{a}(e_i).f_1) a(p_\h(e^i)).f_2\]
which is the classical moment map equation given in \cref{def:classicalmomentmap}.
\end{proof}

\printbibliography

\end{document}